\documentclass[11pt]{article}
\usepackage{amsmath,amssymb,bbm,amsthm}
\usepackage{fullpage}
\usepackage{verbatim}
\usepackage{graphicx}
\usepackage{colortbl,hhline}
\usepackage[normalem]{ulem}
\usepackage[dvipsnames]{xcolor}
\usepackage{tcolorbox}
\usepackage{graphicx}
\usepackage{caption}
\usepackage[explicit]{titlesec}
\usepackage{mathtools}
\usepackage[colorlinks=false,citecolor=blue ]{hyperref}
\usepackage{cleveref}
\usepackage{algpseudocode}
\usepackage{algorithm}
\usepackage{mathrsfs}
\usepackage[shortlabels]{enumitem}
\usepackage{custom}

\usepackage{thmtools,thm-restate}
\newtheorem{theorem}{Theorem}[section]
\newtheorem{proposition}[theorem]{Proposition}
\newtheorem{lemma}[theorem]{Lemma}
\newtheorem{corollary}[theorem]{Corollary}
\theoremstyle{definition}

\theoremstyle{remark}
\newtheorem{remark}[theorem]{Remark}
\theoremstyle{fact}

\title{\textbf{An ensemble Kalman approach to randomized maximum likelihood estimation}}
\author{Pavlos Stavrinides\thanks{Georgia Institute of Technology, School of Computational Science and Engineering, \href{mailto:eqian@gatech.edu}{pstavrinides3@gatech.edu}} \and Elizabeth Qian\thanks{Georgia Institute of Technology, School of Aerospace Engineering and Computational Science and Engineering, \href{mailto:eqian@gatech.edu}{eqian@gatech.edu}}}
\date{}

\begin{document}
\maketitle

\begin{abstract}
    This work proposes ensemble Kalman randomized maximum likelihood estimation, a new derivative-free method for performing randomized maximum likelihood estimation, which is a method that can be used to generate approximate samples from posterior distributions in Bayesian inverse problems. The new method has connections to ensemble Kalman inversion and works by evolving an ensemble so that each ensemble member solves an instance of a randomly perturbed optimization problem. Linear analysis demonstrates that ensemble members converge exponentially fast to randomized maximum likelihood estimators and, furthermore, that the new method produces samples from the Bayesian posterior when applied to a suitably regularized optimization problem. The method requires that the forward operator, relating the unknown parameter to the data, be evaluated once per iteration per ensemble member, which can be prohibitively expensive when the forward model requires the evolution of a high-dimensional dynamical system. We propose a strategy for making the proposed method tractable in this setting based on a balanced truncation model reduction method tailored to the Bayesian smoothing problem. Theoretical results show near-optimality of this model reduction approach via convergence to an optimal approximation of the posterior covariance as a low-rank update to the prior covariance. Numerical experiments verify theoretical results and illustrate computational acceleration through model reduction.
\end{abstract}


\section{Introduction}\label{sec:Intro}

Inverse problems \cite{sanz-alonso_inverse_2023,bach_inverse_2024} appear in numerous disciplines across science, engineering, and medicine, with applications including ocean modeling \cite{rounce_quantifying_2020,yang_tipping_2020,iglesias_well-posed_2014,majumder_freshwater_2021}, medical imaging \cite{kaipio_statistical_2005,chung_numerical_2010,epstein_introduction_2007,scherzer_mathematical_2006}, engineering mechanics \cite{tanaka_inverse_1993}, and many more. Mathematically, the goal of an inverse problem is to infer an unknown parameter $\bv$ from data $\by$ that follow a measurement model that is defined via a forward operator $\bH$, which maps parameters to data, often corrupted by observational noise.

One common approach to inferring $\bv$ from $\by$ is to employ numerical optimization methods to minimize the misfit between the data and the action of the forward map on the parameter estimate \cite{hansen_discrete_2010,mueller_linear_2012}. It is also common to add a regularization term that penalizes distance between the estimate and a reference value that is determined from domain knowledge \cite{engl_regularization_2000}. Classical gradient-based optimization approaches for solving inverse problems include Newton/Gauss-Newton methods \cite{hanke_regularizing_1997,akcelik_parallel_2002,petra_inexact_2012,epanomeritakis_newton-cg_2008,zhu_inversion_2016,worthen_towards_2014}, trust region optimization methods \cite{conn_trust_2000}, and Krylov subspace methods \cite{dai_nonlinear_2011,akcelik_parallel_2002}. Recent widespread use of stochastic gradient algorithms in machine learning has also led to the use of stochastic gradient descent (SGD) in solving inverse problems \cite{ye_optimization_2019,lu_stochastic_2022,jin_convergence_2020}. However, in many large-scale applications, for example in the earth and atmospheric sciences, computing gradients can be intractable.  In such cases, derivative-free methods such as simulated annealing and particle swarm optimization methods~\cite{sen_global_2013,renyuan_combined_1996,fernandez-martinez_particle_2010} are appealing. A large family of derivative-free optimization methods is the ensemble Kalman inversion (EKI) family \cite{calvello_ensemble_2024,chada_iterative_2021,iglesias_ensemble_2013}, which was initially developed in the context of reservoir simulation \cite{emerick_ensemble_2013,emerick_investigation_2013,evensen_analysis_2018,gu_iterative_2007,hanu_subsampling_2023} and climate modeling \cite{bocquet_bayesian_2020,cleary_calibrate_2021,dunbar_calibration_2021,gottwald_supervised_2021,schneider_earth_2017}. The basic EKI algorithm \cite{iglesias_ensemble_2013} evolves an ensemble of particles according to a Kalman filter-type update rule \cite{katzfuss_understanding_2016,kalman_new_1960} to minimize a least squares functional. Analyses of the convergence of EKI \cite{qian_fundamental_2024,schillings_analysis_2017,schillings_convergence_2018,calvello_ensemble_2024,ding_ensemble_2021}, and its stochastic variant \cite{blomker_well_2019,blomker_continuous_2022,blomker_strongly_2018} show that ensemble members converge to  least squares minimizers of the data misfit in a subspace of the span of the ensemble associated with directions that are observable under the forward operator $\bH$. The convergence of the basic EKI algorithm is slow, occurring at a $1/\sqrt{i}$ rate (where $i$ is the iteration index), but acceleration techniques have been proposed, for example, via covariance inflation or localization \cite{chada_convergence_2022,huang_efficient_2022,armbruster_stabilization_2022,blomker_well_2019,tong_localized_2023,liu_dropout_2025,al-ghattas_non-asymptotic_2023}. Tikhonov EKI (TEKI) \cite{chada_tikhonov_2020} extends the basic EKI formulation \cite{iglesias_ensemble_2013} to a Tikhonov-regularized least squares minimization.

EKI methods have also been applied to solve Bayesian inverse problems, which treat all variables as random variables, endow the unknown parameter with a prior density, and use Bayes' theorem to update the prior to a posterior distribution by conditioning on the data~\cite{stuart_inverse_2010}. In the Bayesian setting, the posterior distribution provides an estimate of parameter uncertainty. In this Bayesian view, TEKI \cite{chada_tikhonov_2020} converges to a maximum a posteriori (MAP) estimator \cite{dashti_map_2013}. There also exist a number of EKI-related methods that produce approximate samples from the Bayesian posterior. The ensemble Kalman sampler (EKS) \cite{garbuno-inigo_interacting_2020,nusken_note_2019} uses Langevin dynamics \cite{roberts_exponential_1996} to directly perturb the ensemble and create an interacting particle system that yields approximate posterior samples. Work in \cite{huang_efficient_2022} develops ensemble adjustment Kalman inversion (EAKI), which updates particles deterministically with a pre-multiplier, and ensemble transform Kalman inversion (ETKI), which updates particles deterministically with a post-multiplier. Both EAKI and ETKI produce samples by matching moments of Gaussians. In \cite{huang_iterated_2022}, authors derive unscented Kalman inversion (UKI), and the unscented Kalman sampler (UKS), which apply the unscented Kalman filter to a stochastic dynamical system and compute statistics by deterministic quadrature rules.

Beyond ensemble Kalman methods, other popular methods for solving Bayesian inverse problems include:  variational inference \cite{peterson_mean_1987,opper_variational_2009}, which attempts to minimize the distance from a parameterized approximate density to the posterior, coupling techniques such as transport maps \cite{marzouk_sampling_2016, reich_dynamical_2011,el_moselhy_bayesian_2012}, which aim to transform a random variable from the prior density to an approximate random variable from the posterior density, and classical sampling approaches such as random walk Metropolis and preconditioned Crank-Nicolson, which use ideas from Markov chain Monte Carlo (MCMC) \cite{geyer_practical_1992,roberts_weak_1997,goodman_ensemble_2010,bardsley_mcmc-based_2012,beskos_sequential_2015,del_moral_sequential_2006} to sample directly from the posterior density. In randomized maximum likelihood estimation (RMLE) \cite{chen_ensemble_2012,gu_iterative_2007}, posterior samples can be generated through repeated maximization of a randomly perturbed likelihood function. In the randomize-then-optimize (RTO) method \cite{bardsley_randomize-then-optimize_2014}, samples are generated by repeatedly solving a modified version of the RMLE objective function that leads to a closed-form of the posterior density. However, these optimization-based approaches are gradient-based. In this work, we propose a new ensemble Kalman approach to randomized maximum likelihood estimation that can be used to sample from linear Bayesian posterior distributions. The new method is \textit{derivative-free}, and, unlike the aforementioned EAKI, ETKI, and EKS samplers, our approach is not based on designing particle dynamics to yield posterior samples, but rather on perturbing a least-squares minimization for each particle to solve. 

All of the above methods typically require many forward model evaluations, which is expensive in many applications, particularly when $\bH$ involves solving a high-dimensional system of (differential) equations. This computational challenge can be ameliorated via model reduction. Projection-based model reduction \cite{benner_survey_2015,antoulas_approximation_2005} can be used to approximate a high-dimensional linear dynamical system by projecting the full model onto a low dimensional subspace \cite{antoulas_approximation_2005,antoulas_interpolatory_2020}, and has been used in the inference setting, for example via the proper orthogonal decomposition (POD) \cite{willcox_balanced_2002,lipponen_electrical_2013,lieberman_parameter_2010,nguyen_model_2014}, and the reduced basis method for PDE-constrained optimization~\cite{qian_certified_2017} and data assimilation \cite{karcher_reduced_2018}. Another body of work has sought to exploit the inherent structure of the inverse problem when performing model reduction. The works of \cite{qian_model_2022,konig_time-limited_2023,freitag_inference-oriented_2024} show how balanced truncation \cite{gugercin_survey_2004,moore_principal_1981}, a popular system-theoretic projection-based model reduction method, can be adapted to the problem of inferring the initial condition of a high-dimensional dynamical system. These methods obtain reduced models by projecting the dynamical system onto a subspace spanned by the state directions that are most highly informed by the data relative to their prior uncertainty. This subspace has also been exploited for more efficient sampling in~\cite{spantini_optimal_2015,zahm_certified_2022}, and was shown to define optimal posterior approximations for linear Gaussian inverse problems in \cite{spantini_optimal_2015}.

The contributions of this work are the following:
\begin{enumerate}
    \item We propose ensemble Kalman RMLE, a new method which builds on ensemble Kalman inversion methods to perform randomized maximum likelihood estimation. 
    \item We present linear analysis of our new method. Our results show exponential convergence of (i)~particles to perturbed least squares minimizers, (ii) the ensemble mean to the standard minimum-norm least squares minimizer, and (iii) the ensemble covariance to the Hessian of the optimization evaluated at the minimizer. 
    \item We show that application of our theoretical results to a suitably defined regularized minimization yields exponential convergence of the particles to samples from the posterior distribution of a linear Gaussian Bayesian inverse problem, and of the ensemble mean and covariance to the Bayesian posterior mean and covariance.
    \item For the specific problem of inferring the initial condition of a high-dimensional linear dynamical system, we propose a strategy for reducing the computational cost of our method using the balanced truncation approach of~\cite{qian_model_2022} that yields near-optimal posterior covariance approximations.
    \item We demonstrate numerically our convergence theory, convergence to the Gaussian posterior and computational acceleration via model reduction.
\end{enumerate}

The remainder of this paper is organized as follows. \Cref{sec:Background} mathematically formulates the inverse problem we study and introduces the original EKI algorithm. \Cref{sec:new_method} introduces our new method, followed by a summary of the theoretical properties of the method. \Cref{sec:application_with_BTBI} formulates the Bayesian smoothing problem that we consider and proposes a strategy for accelerating our method in this setting using the balanced truncation approach of~\cite{qian_model_2022}.  Detailed theoretical analysis and proofs are presented in \Cref{sec:LinearAnalysis} followed by numerical results in \Cref{sec:Numerics}. \Cref{sec:Conclusions} concludes.

\section{Background}\label{sec:Background}

This section defines the inverse problem of interest in \Cref{subsec:ProblemForumlation} and introduces the basic ensemble Kalman inversion (EKI) method in \Cref{subsec:EKI_back}.

\subsection{Problem formulation}\label{subsec:ProblemForumlation}

We consider the inverse problem of inferring an unknown parameter of interest $\bv\in\R^d$ from observations $\by\in\R^n$ described by the following measurement model:
    \begin{align}
        \by = \bH(\bv) + \beps,
        \label{eq:inverse_problem}
    \end{align}
    where $\bH:\R^d\to \R^n$ is a called the forward map and $\beps\in\R^n$ represents measurement error. One way to estimate $\bv$ from $\by$ is to minimize the data misfit:
    \begin{align}
        \min_{\bv\in\R^d} (\by - \bH(\bv))^\top\bGamma^{-1}(\by - \bH(\bv)) \equiv\min_{\bv\in\R^d} \|\by - \bH(\bv)\|^2_{\bGamma^{-1}},
        \label{eq:EKI_min}
    \end{align}
where $\bGamma\in \R^{n\times n}$ is a weight matrix. In the statistical inference context, if the noise $\beps$ in~\eqref{eq:inverse_problem} has probability distribution $\beps\sim\mathcal{N}(\bzero,\bGamma)$, then minimizers of~\eqref{eq:EKI_min} are maximum likelihood estimators (MLEs) of $\bv$.

If the optimization problem~\eqref{eq:EKI_min} is ill-posed, and/or if prior information about the unknown $\bv$ is available from domain knowledge, it is common to reformulate the parameter estimation as a regularized minimization, 
\begin{align}
    &\min_{\bv\in\R^d} \big[(\by-\bH(\bv))^\top\bGamma^{-1}(\by - \bH(\bv)) + (\bv-\mupr)^\top\Gpr(\bv - \mupr)\big],\label{eq:TEKI_sum}
\end{align}
which adds a penalty on the distance from the estimate for $\bv$ from a reference value $\mupr\in\R^d$ weighted by positive definite $\Gpr\in\R^{d\times d}$.
Note that the regularized minimization~\eqref{eq:TEKI_sum} can be expressed as a least-squares minimization of the form~\eqref{eq:EKI_min} as follows:
\begin{align}
    \min_{\bv\in\R^d}(\by_{\text{RLS}} - \bH_\text{RLS}(\bv))^\top \bGamma_\text{RLS}^{-1}(\by_{\text{RLS}} - \bH_\text{RLS}(\bv)) \equiv \min_{\bv\in\R^d}\|\by_{\text{RLS}} - \bH_\text{RLS}(\bv)\|^2_{\bGamma_\text{RLS}^{-1}},
    \label{eq:TEKI_min}
\end{align}
where we have defined the following regularized least-squares (RLS) quantities:
\begin{align}
    \by_{\text{RLS}} = 
    \begin{bmatrix}
        \by \\ \mupr
    \end{bmatrix}, &&
    \bH_\text{RLS}(\bv) = 
    \begin{bmatrix}
        \bH(\bv) \\ \bv
    \end{bmatrix},
    &&
    \bGamma_\text{RLS} = 
    \begin{bmatrix}
        \bGamma & \\
        & \Gpr
    \end{bmatrix}.
    \label{eq:TEKI_notation}
\end{align}
The minimizer of~\eqref{eq:TEKI_min} has the following connection to the Bayesian inferential approach: if the parameter $\bv$ is endowed with a Gaussian prior distribution $\mathcal{N}(\mupr,\Gpr)$, then the Bayesian posterior density $\bbP(\bv|\by)$ has the form:
\begin{align}
    \bbP(\bv|\by)\propto \exp\left(-\frac{1}{2} \|\by_\text{RLS} - \bH_\text{RLS}(\bv)\|^2_{\bGamma_\text{RLS}^{-1}}\right),
    \label{eq:posterior_density}
\end{align}
The minimizer of~\eqref{eq:TEKI_min} therefore coincides with the Bayesian maximum a posteriori (MAP) estimator~\cite{dashti_map_2013,chada_tikhonov_2020,stuart_inverse_2010,sanz-alonso_inverse_2023}.

We have thus far considered inverse problem formulations for a general nonlinear forward operator $\bH$. When $\bH$ is \textit{linear}, further connections between the optimization and statistical perspectives are possible.
Note that for linear $\bH$, the standard the norm-minimizing solution of~\eqref{eq:EKI_min} is
    \begin{align*}
        \bv^\star =     \big(\bH^\top\bGamma^{-1}\bH\big)^\dagger \bH^\top\bGamma^{-1}\by \equiv \bH^+\by,
    \end{align*}
    where $\dagger$ denotes the Moore-Penrose pseudoinverse and $\bH^+ = \big(\bH^\top\bGamma^{-1}\bH\big)^\dagger \bH^\top\bGamma^{-1}$ is the weighted pseudoinverse. 
Then, samples from the Gaussian distribution $\mathcal{N}\big(\bv^\star,(\bH^\top\bGamma^{-1}\bH)^\dagger\big)$ may be obtained via \textit{randomized maximum likelihood estimation} (RMLE) \cite{gu_iterative_2007,chen_ensemble_2012}, which solves instances of the following randomized least squares problem:
\begin{align}
    \min_{\bv\in\R^d} \|\by + \beps - \bH\bv\|_{\bGamma^{-1}}^2, 
    \label{eq:RMLE_min}
\end{align}
where one draws one realization of $\beps\sim\mathcal{N}(\bzero,\bGamma)$ per sample.
For linear Gaussian Bayesian inverse problems, applying RMLE to the regularized least-squares operators~\eqref{eq:TEKI_notation} yields samples from the Bayesian posterior distribution $\mathcal{N}(\mupos,\Gpos)$, where
    \begin{align}
        \bGamma_\text{pos} = \big(\bH^\top\bGamma^{-1}\bH + \Gpr^{-1}\big)^{-1}, && \bmu_\text{pos} = \bGamma_\text{pos}\big(\bH^\top\bGamma^{-1}\by + \Gpr^{-1}\mupr\big).
        \label{eq:Gaussian_posterior}
    \end{align}
Note that in the optimization perspective, the posterior mean is the norm-minimizing solution of \eqref{eq:TEKI_min}, i.e., $\mupos = \bv^\star_\text{RLS} =  \bH^+_\text{RLS}\by_\text{RLS}$, and the posterior precision is the Hessian of \eqref{eq:TEKI_min} evaluated at the minimizer, i.e., $\Gpos^{-1} = \bH_\text{RLS}^\top \bGamma_\text{RLS}^{-1}\bH_\text{RLS}$.

\subsection{Ensemble Kalman Inversion (EKI)}\label{subsec:EKI_back}

Ensemble Kalman Inversion (EKI) \cite{iglesias_ensemble_2013} is a derivative-free, iterative method for solving weighted least squares problems of the form \eqref{eq:EKI_min}. 
The method begins with an initial ensemble of size $J$, $\big\{\bv_0^{(1)},\ldots, \bv_0^{(J)}\big\}\in\R^d$, e.g., drawn from a suitable prior distribution, and evolves the ensemble for iterations $i=0,1,2,\ldots$, via the update step,
\begin{align}
    \bv_{i+1}^{(j)} = \bv_{i}^{(j)} + \bK_i \big(\by_i^{(j)} - \bh_i^{(j)}\big), \text{ where } \bK_i = \CovE\big[\bv_i^{(1:J)},\bh_i^{(1;J)}\big]\cdot\big(\CovE\big[\bh_i^{(1:J)}\big] + \bGamma\big)^{-1},
    \label{eq:EKI_update_step}
\end{align}
and $\bh_i^{(j)} = \bH\big(\bv_i^{(j)}\big)$, for all $j = 1,\ldots,J
$. Throughout this work, we use $\E$ and $\Cov$ to denote the true expected value and covariance, respectively, and $\EE$ and $\CovE$ to denote the empirical ensemble mean and covariance, respectively. That is, given$\{\ba^{(j)}\}_{j=1}^J$ and $\{\bb^{(j)}\}_{j=1}^J$, we define:
    \begin{align*}
        \EE\big[\ba^{(1:J)}\big] = \frac{1}{J}\sum_{j=1}^J \ba^{(j)}, && \CovE \big[\ba^{(1:J)},\bb^{(1:J)}\big] = \frac{1}{J-1}\sum_{j=1}^J\big(\ba^{(j)}- \EE\big[\ba^{(j)}\big]\big)\big(\bb^{(j)} - \EE\big[\bb^{(j)}\big]\big)^\top,
    \end{align*}
and let $\CovE\big[\ba^{(1:J)}\big] = \CovE \big[\ba^{(1:J)},\ba^{(1:J)}\big]$ for brevity. 
The observation $\by_i^{(j)}$ used to define the update step~\eqref{eq:EKI_update_step} can be defined in two ways, yielding two variants of the basic EKI algorithm:
\begin{subequations}
    \begin{align}
        &\text{Deterministic EKI: }&&\by_i^{(j)} = \by,\label{eq:DetEKI}\\
         &\text{Stochastic EKI: }&& \by_i^{(j)} = \by + \beps_i^{(j)}, \text{ with }\beps_i^{(j)}\sim\mathcal{N}(\bzero,\bSigma), \text{ i.i.d.}\label{eq:StochEKI}
    \end{align}
\end{subequations}
In stochastic EKI, it is common to set $\bSigma=\bGamma$, so that the perturbations in~\eqref{eq:StochEKI} are drawn from a distribution that is consistent with the statistical interpretation of the noise in~\eqref{eq:inverse_problem}. 
\Cref{alg:EKI} summarizes both the deterministic and stochastic versions of basic EKI. 
We emphasize that the EKI algorithm is applicable to nonlinear $\bH$ and is adjoint- and derivative-free, only requiring forward evaluations of $\bH$.
The work~\cite{chada_tikhonov_2020} shows that EKI can be applied equally well to solve regularized least-squares problems of the form \eqref{eq:TEKI_min} by simply running \Cref{alg:EKI} using the RLS operators~\eqref{eq:TEKI_notation}.

For linear $\bH$, particles in both the deterministic and stochastic variants of the basic EKI algorithm have been shown to collapse to the minimum norm solution $\bv^\star$ of \eqref{eq:EKI_min} in a subspace of the range of the ensemble that is associated with observable directions under $\bH$, at a $1/\sqrt{i}$ rate~\cite{qian_fundamental_2024}. Thus, in the $i\to\infty$ limit, the algorithm is viewed as an optimizer and not as a sampler. Several EKI-related approaches have been proposed for sampling from Bayesian posteriors, for example by integrating the continuous-time limit of the EKI iteration up to $t=1$~\cite{calvello_ensemble_2024}, designing a system based on Langevin dynamics that converges to the posterior~\cite{garbuno-inigo_interacting_2020}, or using deterministic quadrature rules to compute statistics~\cite{huang_iterated_2022}. In this work, we propose a new sampling algorithm that combines classical EKI with randomized maximum likelihood estimation.

\begin{algorithm}[htb!]
    \caption{Ensemble Kalman Inversion (EKI)}\label{alg:EKI}
    \textbf{Input:} initial ensemble $\big\{\bv_0^{(1)},\ldots, \bv_0^{(J)}\big\}\subset\R^d$, forward operator $\bH:\R^d\to\R^n$, observations $\by\in\R^n$, observation perturbation covariance $\bSigma\in\R^{n\times n}$, weight matrix $\bGamma\in\R^{n\times n}$

    For deterministic EKI, set $\bSigma=\bzero$. For stochastic EKI, a common choice is $\bSigma = \bGamma$.
    \begin{algorithmic}[1]
        \For{$i=0,1,2,\ldots$}
            \State{Evolve ensemble: $\bh_i^{(j)} = \bH\big(\bv_i^{(j)}\big)$ for $j=1,\ldots, J$}
            \State{Compute sample covariances: $\CovE\big[\bh_i^{(1:J)}\big]$ and $\CovE\big[\bv_i^{(1:J)},\bh_i^{(1;J)}\big]$}
            \State{Compute ensemble Kalman gain: $\bK_i = \CovE\big[\bv_i^{(1:J)},\bh_i^{(1;J)}\big]\cdot\big(\CovE\big[\bh_i^{(1:J)}\big] + \bGamma\big)^{-1}$}
            \State{Sample $\big\{\beps_i^{(1)}, \ldots, \beps_i^{(J)}\big\}\sim\mathcal{N}(\bzero,\bSigma)$ i.i.d.}
            \State{Perturb observations: $\by_i^{(j)} = \by + \beps_i^{(j)}$, for $j = 1,\ldots,J$}
            \State{Update particles: $\bv_{i+1}^{(j)} = \bv_i^{(j)} + \bK_i\big(\by_i^{(j)} - \bh_i^{(j)}\big)$ for $j = 1,\ldots, J$}
            \If{converged}
                \State{\Return ensemble mean $\EE\big[\bv_{i+1}^{(1:J)}\big]$}
            \EndIf{}
        \EndFor{}
    \end{algorithmic}
\end{algorithm}

\section{Ensemble Kalman Randomized Maximum Likelihood Estimation}\label{sec:new_method}

This section introduces our new method (\Cref{subsec:EKRMLE}) and summarizes its convergence properties (\Cref{subsec:idealized}).

\subsection{Ensemble Kalman Randomized Maximum Likelihood Estimation}\label{subsec:EKRMLE}

We propose a new ensemble Kalman method based on randomized maximum likelihood estimation. Our method is a variant of the basic EKI algorithm with $\by_i^{(j)}$ in the update step~\eqref{eq:EKI_update_step} defined as follows: 
\begin{align}
    \by_i^{(j)} = \by^{(j)} \equiv \by + \beps^{(j)}, \quad\text{where }\beps^{(j)}\sim \mathcal{N}(\bzero,\bGamma),
    \label{eq:EKIRMLE_update_step}
\end{align}
for all $j=1,\ldots, J$. We will show that for linear $\bH$, the $j$-th ensemble member therefore solves 
\begin{align}
\min_{\bv\in\R^d}\norm{\by^{(j)}- \bH(\bv)}_{\bGamma^{-1}}^2,
    \label{eq:EKIRMLE_min}
\end{align}
which is an instance of the randomized least-squares problem solved by randomized maximum likelihood estimation~\eqref{eq:RMLE_min}.
 We call our method \textit{ensemble Kalman randomized maximum likelihood estimation}, and summarize its implementation in \Cref{alg:our_method}.

One key difference between our method and basic EKI is in the data perturbation used to define the update step. Deterministic EKI (\Cref{alg:EKI} with the choice~\eqref{eq:DetEKI}) uses a zero perturbation for all particles for all iterations. Stochastic EKI (\Cref{alg:EKI} with the choice~\eqref{eq:StochEKI}) perturbs the data with a perturbation that is independently drawn for each particle \textit{at every iteration}. In our method (\Cref{alg:our_method}), the update step for each particle uses a perturbation that is independently drawn for each particle, but \textit{constant across iterations}. Another key difference between our method and basic EKI is that EKI (\Cref{alg:EKI}) returns a point estimate, namely, the ensemble mean, whereas our method (\Cref{alg:our_method}) returns a collection of samples which allows mean and covariance estimates to be computed. 

\begin{algorithm}[htb!]
    \caption{Ensemble Kalman Randomized Maximum Likelihood Estimation}\label{alg:our_method}
    \textbf{Input:} initial ensemble $\big\{\bv_0^{(1)},\ldots, \bv_0^{(J)}\big\}\subset\R^d$, forward operator $\bH:\R^d\to\R^n$, observations $\by\in\R^n$, weight matrix $\bGamma\in\R^{n\times n}$
    \begin{algorithmic}[1]
        \State{Sample $\big\{\beps^{(1)}, \ldots, \beps^{(J)}\big\}\sim\mathcal{N}(\bzero,\bGamma)$ i.i.d.}
        \State{Create perturbed observations: $\by^{(j)} = \by + \beps^{(j)}$, for $j = 1,\ldots,J$}
        \For{$i=0,1,2,\ldots$}
            \State{Evolve ensemble: $\bh_i^{(j)} = \bH\big(\bv_i^{(j)}\big)$ for $j=1,\ldots, J$}
            \State{Compute sample covariances: $\CovE\big[\bh_i^{(1:J)}\big]$ and $\CovE\big[\bv_i^{(1:J)},\bh_i^{(1;J)}\big]$}
            \State{Compute ensemble Kalman gain: $\bK_i = \CovE\big[\bv_i^{(1:J)},\bh_i^{(1;J)}\big]\cdot\big(\CovE\big[\bh_i^{(1:J)}\big] + \bGamma\big)^{-1}$}
            \State{Update particles: $\bv_{i+1}^{(j)} = \bv_i^{(j)} + \bK_i\big(\by^{(j)} - \bh_i^{(j)}\big)$ for $j = 1,\ldots, J$}
            \If{converged}
                \State{\Return ensemble $\big\{\bv_{i+1}^{(1)},\ldots, \bv_{i+1}^{(J)}\big\}$, ensemble mean $\EE\big[\bv_{i+1}^{(1:J)}\big]$ and covariance $\CovE\big[\bv_{i+1}^{(1:J)}\big]$}
            \EndIf{}
        \EndFor{}
    \end{algorithmic}
\end{algorithm}

\subsection{Linear mean field analysis}\label{subsec:idealized}

We now summarize our main theoretical results concerning convergence of our proposed method, deferring detailed technical proofs and intermediate results to \Cref{sec:LinearAnalysis}. Without loss of generality, we assume that the initial ensemble is drawn from a distribution with mean $\bmu_0\in\R^d$ and positive (semi-)definite covariance $\bCTilde_0\in\R^{d\times d}$. This allows us to interpret an arbitrary initial ensemble as being drawn from a distribution with mean and covariance given by the initial sample mean and covariance. 
We also assume that the data perturbations $\beps^{(j)}$ are independent of the initial ensemble. 

For $j = 1,2,\ldots, J$, let $\bv_\star^{(j)}$ denote the norm-minimizing solution of the perturbed least-squares minimization~\eqref{eq:EKIRMLE_min}, i.e.,
    \begin{align*}
        \bv_\star^{(j)} = \big(\bH^\top\bGamma^{-1}\bH\big)^\dagger \bH^\top\bGamma^{-1}\by^{(j)} \equiv\bH^+\by^{(j)},
    \end{align*}
    and denote by $\bomega_{i+1}^{(j)} = \bv_{i+1}^{(j)} -\bv_\star^{(j)}$ the residual between the $j$-th particle and its corresponding minimum-norm least squares solution. Let $\bC_i\equiv \CovE\big[\bv_i^{(1:J)}\big]$. Then:
\begin{proposition}\label{prop:state_residual}
        The $j$-th particle residual satisfies
        \begin{align}
    \bomega_{i+1}^{(j)} = \bM_i\bomega_i^{(j)} = \bM_{i:0}\bomega_0^{(j)},
        \label{eq:state_residual}
\end{align}
where we define $\bM_i \equiv \bI-\bK_i\bH $ and, for $i\geq k$, we introduce:
\begin{align}
    \bM_{i:k} \equiv \bM_i\bM_{i-1}\cdots \bM_{k+1}\bM_k = \prod_{\ell=k}^i \bM_\ell.
    \label{eq:compound_bbM_only}
\end{align}
\end{proposition}

\begin{proof}
For linear $\bH$, note that $\CovE\big[\bH\bv_i^{(1:J)}\big] = \bH\bC_i\bH^\top$, and $\CovE\big[\bv_i^{(1:J)}, \bH\bv_i^{(1:J)}\big] =\bC_i\bH^\top$, so that $\bK_{i} = \bC_i\bH^\top\big(\bH\bC_i\bH^\top + \bGamma\big)^{-1}$. Then, the update rule in \cref{eq:EKIRMLE_update_step} becomes:
        \begin{align}
            \bv_{i+1}^{(j)} &= \bv_i^{(j)} + \bK_i\big(\by^{(j)} - \bH\bv_i^{(j)}\big)= (\bI - \bK_{i}\bH)\bv_i^{(j)} + \bK_{i}\by^{(j)} = \bM_i \bv_i^{(j)} + \bK_i\by^{(j)}.
            \label{eq:new_update_step}
        \end{align}
Substituting \eqref{eq:new_update_step} into the residual $\bomega_{i+1}^{(j)}$ yields:
    \begin{align}
        \bomega_{i+1}^{(j)} &= \bv_{i+1}^{(j)} - \bv_\star^{(j)} = \bM_i\bv_i^{(j)} + \bK_i\by^{(j)} - \bv_\star^{(j)}\nonumber \\
        &=\bM_i\big(\bv_i^{(j)} - \bv_\star^{(j)}\big) + \bM_i\bv_\star^{(j)} - \bv_\star^{(j)} + \bK_i\by^{(j)}
         = \bM_i\bomega_i^{(j)} - \bK_i\big(\bH\bv_\star^{(j)} - \by^{(j)}\big).
         \label{eq:residual_part1}
    \end{align}
    To simplify the above expression, we note that via Woodbury identity:
    \begin{align*}
        \big(\bI + \bC_i\bH^\top\bGamma^{-1}\bH\big)^{-1} = \bI - \bC_i\bH^\top\big(\bH\bC_i\bH^\top + \bGamma\big)^{-1}\bH = \bI-\bK_i\bH\equiv\bM_i.
    \end{align*}
    Another application of the Woodbury identity yields:
    \begin{align*}
        \big(\bI + \bC_i\bH^\top\bGamma^{-1}\bH\big)^{-1} = \bI - \bC_i\bH^\top\bGamma^{-1}\big(\bI + \bH\bC_i\bH^\top\bGamma^{-1}\big)^{-1}\bH
    \end{align*}
Combining the above allows us to re-express the Kalman gain $\bK_i$ as:
\begin{align*}
    \bM_i\bC_i\bH^\top\bGamma^{-1} &= \bC_i\bH^\top\bGamma^{-1} - \bC_i\bH^\top\bGamma^{-1}\big(\bI + \bH\bC_i\bH^\top\bGamma^{-1}\big)^{-1}\bH\bC_i\bH^\top\bGamma^{-1}\\
    &= \bC_i\bH^\top\big(\bGamma^{-1} - \bGamma^{-1}\big(\bI + \bH\bC_i\bH^\top\bGamma^{-1}\big)^{-1}\bH\bC_i\bH^\top\bGamma^{-1}\big) = \bC_i\bH\big(\bH\bC_i\bH^\top + \bGamma\big)^{-1} = \bK_i.
\end{align*}
Finally, we notice that $\bH\bv_\star^{(j)} - \by^{(j)}$ in~\eqref{eq:residual_part1} is the least squares misfit of the norm-minimizing solution, and thus $\big(\bH\bv_\star^{(j)} - \by^{(j)}\big)\in\kernel\big(\bH^\top\bGamma^{-1}\big)$, so $\bK_i\big(\bH\bv_\star^{(j)} - \by^{(j)}\big) = \bM_i\bC_i\bH^\top\bGamma^{-1}\big(\bH\bv_\star^{(j)} - \by^{(j)}\big) = \bzero$, and $\bomega_{i+1}^{(j)} = \bM_i\bomega_i^{(j)}$. The second equality of~\eqref{eq:state_residual} follows directly.
\end{proof}

Note that in \eqref{eq:state_residual}, the misfit iteration matrices $\bM_i$ are random variables due to their dependence on $\bC_i$, which depends on the initial ensemble. We will therefore analyze the particle behavior in the mean-field limit where ensemble size $J\to\infty$. We denote by $\bCTilde_i$ the mean field ensemble covariance, and  similarly denote by $\bMTilde_i$ the mean-field limit of $\bM_i$, corresponding to
\begin{align}
    \bMTilde_i = \big(\bI+\bCTilde_i\bH^\top\bGamma^{-1}\bH\big)^{-1},
    \label{eq:meanfield_compound_bbM_only}
\end{align}
and define $\bMTilde_{i:0}=\bMTilde_i\bMTilde_{i-1}\cdots\bMTilde_0$. These mean-field operator limits define an iteration which governs the behavior of the mean-field residual, $\bomegaTilde_i^{(j)}$, as follows:

\begin{align}
        \bomegaTilde_{i+1}^{(j)} = \bMTilde_i\bomegaTilde_i^{(j)} = \bMTilde_{i:0}\bomega_0^{(j)}.
        \label{eq:idealized_state_misfit}
    \end{align} 
Understanding the evolution of the particle residual therefore requires an understanding of the spectral properties of the map $\bMTilde_i$, which are characterized by the following eigenvalue problem,
    \begin{align}
        \bCTilde_{i}\bH^\top\bGamma^{-1}\bH\bu_{\ell, i} = \lambda_{\ell, i}\bu_{\ell, i}, \quad \ell = 1,\ldots, d.
        \label{eq:statespace_GEV_alt}
    \end{align}
Denote by $\bvTilde_i^{(j)}$ the mean-field particle corresponding to the mean-field residual $\bomegaTilde_i^{(j)} = \bvTilde_i^{(j)}-\bv_\star^{(j)}$, and note that~\eqref{eq:idealized_state_misfit} implies:
\begin{align}
       \bvTilde_{i+1}^{(j)} =\bMTilde_{i:0}\bv_0^{(j)} + \big(\bI - \bMTilde_{i:0}\big)\bv_\star^{(j)} = \bMTilde_{i:0}\bv_0^{(j)} + \big(\bI - \bMTilde_{i:0}\big)\big(\bH^\top\bGamma^{-1}\bH\big)^\dagger\bH^\top\bGamma^{-1}\by^{(j)}.
       \label{eq:mean_field}
    \end{align}
Then, the mean-field covariance satisfies:
\begin{align}
        \bCTilde_{i+1} = \Cov\big[\bvTilde_{i+1}^{(j)}\big] &= \bMTilde_{i:0}\bCTilde_0\bMTilde_{i:0}^{\top} + \big(\bI - \bMTilde_{i:0}\big)\big(\bH^\top\bGamma^{-1}\bH\big)^\dagger\big(\bI - \bMTilde_{i:0}\big)^\top.
        \label{eq:idealized_cov}
    \end{align}
We find that under~\eqref{eq:idealized_cov}, the eigenvectors of~\eqref{eq:statespace_GEV_alt} are constant across iterations, and that the corresponding eigenvalues can be given in terms of the sequence of preceding eigenvalues:

\begin{proposition}\label{prop:GEV_statespace_alt}
            The eigenvectors of~\eqref{eq:statespace_GEV_alt} are constant, satisfying $\bu_{\ell, i+1} = \bu_{\ell, i}$ for all $i\geq 0$. The eigenvalue associated with the $\ell$-th eigenvector satisfies:
            \begin{align}
                \lambda_{\ell, i+1} = 1 - 2\prod_{k=0}^i(1+\lambda_{\ell, k})^{-1} + (1+\lambda_{\ell, 0})\prod_{k=0}^i(1+\lambda_{\ell, k})^{-2}.
                \label{eq:idealized_eigen_evolution}
            \end{align}
        \end{proposition}   
Notice that all eigenvalues of~\eqref{eq:statespace_GEV_alt} are non-negative because nonzero eigenvalues of $\bCTilde_i\bH^\top\bGamma^{-1}\bH$ must be equal to nonzero eigenvalues of $\bGamma^{-\frac12}\bH\bCTilde_i\bH^\top\bGamma^{-\frac12}$, with zero eigenvalues differing only in multiplicity. From \eqref{eq:idealized_eigen_evolution}, we see that zero eigenvalues remain zero and strictly positive eigenvalues remain strictly positive across all iterations. We show in \Cref{sec:LinearAnalysis} that for all $i\geq 1$, nonzero eigenvalues of~\eqref{eq:statespace_GEV_alt} are monotone increasing across iterations and bounded above by $1$, and further, that:

\begin{proposition}\label{prop:limit}
        Nonzero eigenvalues of~\eqref{eq:statespace_GEV_alt} converge exponentially fast, i.e.,:
        \begin{align*}
            \lim_{i\to\infty}\lambda_{\ell, i+1} = 1 \text{, with } 1-\lambda_{\ell, i+1} \leq 2{(1+\lambda_{\ell, 0}})^{-1} e^{-(i-1)\gamma_\ell},
        \end{align*}
        where $\gamma_\ell = \lambda_{\ell, 0}(1+2\lambda_{\ell, 0})^{-1}$.
    \end{proposition}

Let $r$ denote the number of strictly positive (nonzero) eigenvalues of \eqref{eq:statespace_GEV_alt}. At $i=1$, arrange the $r$ strictly positive eigenvalues (and their respective eigenvectors) in non-increasing order $\lambda_{1,1}\geq \lambda_{2,1}\geq \cdots \geq \lambda_{r,1}$.  \Cref{cor:fixed_order_of_eigenvalues} establishes that this ordering is preserved across all iterations $i\geq 1$. Since the eigenvectors of \eqref{eq:statespace_GEV_alt} are constant, we drop the iteration index $i$ on the eigenvectors and consider the set $\{\bu_\ell\}_{\ell=1}^r$, normalized so that $\bU = [\bu_1,\ldots, \bu_r]$ satisfies $\bU^\top\bH^\top\bGamma^{-1}\bH\bU = \bI_r$. Let $\bLambda_{1:r}^{(i)} = \diag(\lambda_{1,i},\ldots, \lambda_{r,i})$. Then, $\bCTilde_i\bH^\top\bGamma^{-1}\bH\bU = \bU\bLambda_{1:r}^{(i)}$, and $\range(\bU)$ is an invariant subspace under $\bMTilde_{i}$. This subspace and its complement may be defined in terms of spectral projectors of $\bMTilde_{i:0}$, as follows: \begin{proposition}\label{prop:bbM_spectral_projectors}
    Let $\bP = \bU\bU^\top\bH^\top\bGamma^{-1}\bH$, and $\bS = \bI - \bP$. Then, $\bP$ and $\bS$ are complementary spectral projectors of $\bMTilde_{i:0}$. That is, $\bP^2 = \bP$, $\bS^2=\bS$ and $\bP\bMTilde_{i:0} = \bMTilde_{i:0}\bP$, $\bS\bMTilde_{i:0} = \bMTilde_{i:0}\bS$, with $\bP\bS = \bS\bP= \bzero$ and $\bP + \bS  =\bI_d$.
\end{proposition}

The spectral projectors $\bP$ and $\bS$ introduced in \Cref{prop:bbM_spectral_projectors} define two complementary subspaces of $\R^d$, $\textsf{ran}(\bP)$ and $\textsf{ran}(\bS)$. Note that $\textsf{ran}(\bP)=\textsf{ran}(\bU)$ is the span of the eigenvectors of~\eqref{eq:statespace_GEV_alt} associated with nonzero eigenvalues; these directions lie in $\textsf{ran}(\bCTilde_i)$ and also must not lie within $\textsf{ker}(\bH)$. The subspace $\textsf{ran}(\bP)$ may thus be interpreted as being `populated' by the ensemble while being observable under $\bH$. 
\Cref{prop:bbM_spectral_projectors} allows the mean field residual $\bomegaTilde_{i+1}^{(j)}$ \eqref{eq:idealized_state_misfit} to be decomposed into two components, one lying in the observable and populated space $\textsf{ran}(\bP)$, and the other in the complementary subspace $\textsf{ran}(\bS)$:
    \begin{align*}
        \bomegaTilde_{i+1}^{(j)} = \bP\bMTilde_{i:0}\bomega_0^{(j)} + \bS\bMTilde_{i:0}\bomega_0^{(j)}=\bMTilde_{i:0}\bP\bomega_0^{(j)} + \bMTilde_{i:0}\bS\bomega_0^{(j)}.
    \end{align*}

Our next result shows that the residual component in $\textsf{ran}(\bP)$ converges to zero, while the component in $\textsf{ran}(\bS)$ remains constant for all iterations. 
We note that the analysis of basic ensemble Kalman inversion in~\cite{qian_fundamental_2024} further divides the non-convergent subspace $\textsf{ran}(\bS)$ into two subspaces associated with (i) observable but populated directions and (ii) unobservable directions. While such a division is also possible in our case, we choose to present results for just the convergent space $\textsf{ran}(\bP)$ and the entire non-convergent space $\textsf{ran}(\bS)$ to simplify the presentation.

\begin{theorem}\label{thm:state_misfit_convergence}
        Let $\gamma=\min_{1\leq \ell\leq r} \gamma_\ell$, where $\gamma_\ell$ is defined as in \Cref{prop:limit}. For all particles $j = 1,\ldots, J$, the following hold for all $i\geq 0$:
        \begin{enumerate}[(a)]
            \item $\bS\bomegaTilde_{i+1}^{(j)} = \bS\bomega_0^{(j)}$
            \item $\bP\bomegaTilde_{i+1}^{(j)} \leq e^{-(i-1)\gamma}\bP\bomega_0^{(j)}$.
        \end{enumerate}
    \end{theorem}

\Cref{thm:state_misfit_convergence} establishes the existence of a limit for $\bomegaTilde_{i+1}^{(j)}$, and exponential convergence. The next result identifies the limit for each particle and, consequently, the limit for the ensemble mean and covariance.

\begin{theorem}\label{thm:master_convergence}
In the mean field limit, as $i\to\infty$, for all particles $j=1,\ldots, J$, it holds that
\begin{enumerate}[(a)]
    \item $\bvTilde_{i+1}^{(j)}\to\bvTilde_\infty^{(j)} = \bP\bv_\star^{(j)}  + \bS\bv_0^{(j)}$,
    \item $\E\big[\bvTilde_{i+1}^{(j)}\big]\to\E\big[\bvTilde_\infty^{(j)}\big] = \bP\bv^\star +\bS\bmu_0$,
    \item $\Cov\big[\bvTilde_{i+1}^{(j)}\big]\to\Cov\big[\bvTilde_\infty^{(j)}\big] = \bP \big(\bH^\top\bGamma^{-1}\bH\big)^\dagger\bP^\top + \bS\bCTilde_0\bS^\top$.
\end{enumerate}
    
\end{theorem}

\begin{corollary}\label{cor:convergence_in_distribution}
    \Cref{thm:master_convergence} establishes that our method performs randomized maximum likelihood estimation in the $\bP$ subspace. Namely, we have that $\bP\bvTilde_{i+1}^{(j)}\to\bP\bvTilde_\infty^{(j)} = \bP\bv_\star^{(j)}$, i.e., mean field particles become randomized maximum likelihood estimators in the $\bP$ subspace, and we also have convergence in distribution of $\bP\bvTilde_{i+1}^{(j)}$ to $\mathcal{N}\big(\bP\bv^\star, \bP\big(\bH^\top\bGamma^{-1}\bH\big)^\dagger\bP^\top\big)$, i.e., we converge to the projection of the RMLE distribution under $\bP$.
\end{corollary}

The results of \Cref{thm:master_convergence,cor:convergence_in_distribution} show exponential convergence (guaranteed by \Cref{thm:state_misfit_convergence}) of mean field particles to randomized maximum likelihood estimators distributed according to the RMLE distribution, in the $\bP$ subspace. In particular, these results contrast with the behavior of stochastic and deterministic versions of basic EKI (\Cref{alg:EKI}), where \textit{all} particles exhibit ensemble collapse to \textit{same} limit in the $\bP$ subspace \cite{qian_fundamental_2024} at a slow $1/\sqrt{i}$ rate of convergence. In the Bayesian inference setting, we define a suitably regularized least-squares problem~\eqref{eq:TEKI_min}, and apply our method (\Cref{alg:our_method}) with $\by_\text{RLS}, \bH_\text{RLS}$ and $\bGamma_\text{RLS}$ to solve the Tikhonov-regularized optimization \eqref{eq:TEKI_min}. Application of \Cref{thm:master_convergence} to this regularized problem shows that the proposed method yields samples from the linear Gaussian Bayesian posterior:

\begin{corollary}\label{cor:application_to_BIP}
        In the mean field limit, as $i\to\infty$, for all particles $j=1,\ldots, J$, it holds that
    \begin{enumerate}[(a)]
        \item $\E\big[\bP\bvTilde_{i+1}^{(j)}\big]\to \bP\mupos$,
        \item $\Cov\big[\bP\bvTilde_{i+1}^{(j)}\big]\to\bP\Gpos\bP^\top$.
    \end{enumerate}
\end{corollary}
\Cref{cor:application_to_BIP} establishes that $\bP\bvTilde_{i+1}^{(j)}$ converges in distribution to $\mathcal{N}\big(\bP\mupos,\bP\Gpos\bP^\top\big)$, which is the projection of exact Gaussian posterior of \cref{eq:Gaussian_posterior} onto the $\bP$ subspace. This means that, in the $\bP$ subspace, mean field particles become samples from the Gaussian posterior. We emphasize that results are derived for the mean-field limit of the iteration: we will demonstrate in \Cref{sec:Numerics} that if the ensemble size is insufficiently large, the true algorithmic iterations will lie far from their mean field counterparts.

\begin{remark}
While the main results presented in this section concern the state space residual, similar results can be shown regarding the behavior of the observation space data misfit, $\btheta_i^{(j)} = \bH\bv_i^{(j)} - \by^{(j)}$. This analysis can be found in \Cref{sec:LinearAnalysis}.
\end{remark}

\section{Model reduction for ensemble Kalman RMLE for the linear Bayesian smoothing problem}\label{sec:application_with_BTBI}

We introduce in \Cref{subsec:LTI_IP} the \textit{smoothing problem}, a special case of the inverse problem~\eqref{eq:inverse_problem} describing the inference of an unknown initial condition of a linear dynamical system from noisy observations taken after the initial time. Ensemble methods for solving such inverse problems can be computationally expensive because the dynamical system must be evolved for each particle at every iteration. To reduce the computational cost in this setting, we describe in \Cref{subsec:BTBI} a model reduction approach from~\cite{qian_model_2022} tailored to this smoothing problem. \Cref{subsec:BTEKI_result} proposes a strategy for reducing the computational cost of our proposed ensemble Kalman RMLE approach using this model reduction approach, and describes approximation guarantees for this strategy, building on our results in \Cref{subsec:idealized}.

\subsection{Inverse problem formulation}\label{subsec:LTI_IP}

We consider a linear dynamical system with linear output:
    \begin{align}
    \begin{split}
        \dot{\bx}(t) &= \bA \bx(t),\\
        \bcaly(t) &= \bF \bx(t),
    \end{split}
    \label{eq:LTI_system}
    \end{align}
    where $\bx(t)\in \R^d$, $\bA\in\R^{d\times d}$, $\bF\in\R^{d_{\text{out}}\times d}$, and the initial state $\bx(0)=\bv$ is unknown. We take noisy measurements, $\by_k$, of the system output $\bcaly$ at positive times $t_1<t_2<\cdots <t_m$ according to the following measurement model:
    \begin{align}
        \by_k \equiv \bF\bx(t_k) + \bheta_k = \bF e^{\bA t_k}\bv + \bheta_k,
        \label{eq:LTI_measurement}
    \end{align}
    where $\bheta_k\sim\mathcal{N}(\bzero,\bGamma_\eta)$, $\bGamma_\eta\in \R^{d_\text{out}\times d_\text{out}}$. Our goal is to infer the unknown initial state $\bv$ of the system~\eqref{eq:LTI_system} from the noisy measurements~\eqref{eq:LTI_measurement}. This defines a linear measurement model of the form~\eqref{eq:inverse_problem} where:
    \begin{align}
        \by=
        \begin{bmatrix}
            \by_1 \\\vdots\\ \by_m
        \end{bmatrix}\in \R^{n}, &&
        \bH=
            \begin{bmatrix}
                \bF e^{\bA t_1} \\ \vdots \\ \bF e^{\bA t_m}
            \end{bmatrix} \in \R^{n\times d}, &&
            \bGamma =
            \begin{bmatrix}
                \bGamma_\eta & &\\
                & \ddots & \\
                & & \bGamma_\eta
            \end{bmatrix} \in \R^{n\times n},
            \label{eq:LTI_operators}
    \end{align}
    with $n= m\cdot d_\text{out}$. We endow $\bv$ with the centered Gaussian prior $\mathcal{N}(\bzero,\Gpr)$. Then, the posterior mean and covariance are given by \cref{eq:Gaussian_posterior} (with $\mupr=\bzero$), whose mean is the minimizer of the Tikhonov-regularized optimization given by \cref{eq:TEKI_min,eq:TEKI_notation}.

In many application settings, solving Bayesian smoothing problems of this form is computationally challenging because $\bH$ is only implicitly available through simulating a high-dimensional linear system. In our proposed ensemble Kalman RMLE approach, the system must be evolved once per particle per iteration, which may be  computationally expensive. To combat this, we next introduce a balanced truncation (BT) model reduction approach from~\cite{qian_model_2022} that is tailored to this smoothing problem.

\subsection{Model reduction via balanced truncation for Bayesian inference}\label{subsec:BTBI}

Balanced truncation \cite{gugercin_survey_2004,moore_principal_1981} is a system-theoretic method for projection-based model reduction that was adapted to linear Bayesian smoothing problems in \cite{qian_model_2022}. The key idea is to project the high-dimensional dynamical system~\eqref{eq:LTI_system} onto state directions that maximize the following Rayleigh quotient:
\begin{align}
    \frac{\bz^\top \bQ\bz}{\bz^\top \Gpr^{-1}\bz},
    \label{eq:Rayleigh}
\end{align}
where $\bQ$ is defined as:
\begin{align}
    \bQ \equiv \int_0^\infty e^{\bA^\top t}\bF^\top \bGamma_\eta^{-1}\bF e^{\bA t}\,\mathrm{d}t\in \R^{d\times d},
    \label{eq:BTBI_Gramians}
\end{align}
so that $\bz^\top\bQ\bz$ can be interpreted as an energy based on integrating, over all positive time, the squared Mahalanobis distance between the equilibrium of the dynamical system and noisy outputs. Directions that maximize \eqref{eq:Rayleigh} are given by the leading eigenvectors of the following generalized eigenvalue problem:
\begin{align*}
    \bQ\bz_\ell = \delta_\ell\Gpr^{-1}\bz_\ell,
\end{align*}
and can be computed as follows. Let $\bGamma_\text{pr} = \bR\bR^\top, \bQ = \bL\bL^\top$ (e.g., via Cholesky factorization) and let $\bL^\top\bR = \bPhi\bXi\bPsi^\top$ be the singular value decomposition (SVD) of $\bL^\top\bR$. Let $\bcalU = \bR\bPsi\bXi^{-\frac{1}{2}}$ and note that $\bcalU$ is invertible with $\bcalU^{-1} = \bXi^{-\frac{1}{2}}\bPhi^\top\bL^\top$. Generalized eigenvectors $\bz_\ell$ of $\big(\bQ,\Gpr^{-1}\big)$ are columns of $\bR\bPsi$ and become columns of $\bcalU$ under the renormalization $\bXi^{-\frac{1}{2}}$. Denote by $\bPhi_\rho,\bPsi_\rho$ the first $\rho$ columns of $\bPhi$ and $\bPsi$ respectively, and let $\bXi_\rho$ be the upper $\rho\times \rho$ block of $\bXi$. Then, define
    \begin{align}
       \bcalV_\rho^\top = \bXi_\rho^{-\frac{1}{2}}\bPhi_\rho^\top\bL^\top, && \bcalU_\rho = \bR \bPsi_\rho\bXi_\rho^{-\frac{1}{2}},
        \label{eq:BTBI_reduced_bases}
    \end{align}
    where $\bcalV_\rho^\top\bcalU_\rho = \bI_\rho$. The order-$\rho$ balanced truncation approximation of~\eqref{eq:LTI_system} is a Petrov-Galerkin projection given by restricting the state $\bx$ to lie in $\range(\bcalU_\rho)$ and by enforcing the system~\eqref{eq:LTI_system} on $\kernel\big(\bcalV_\rho^\top\big)$. This results in the reduced-order approximation of \eqref{eq:LTI_system},
    \begin{align}
    \begin{split}
        \dot{\bhatx}(t) &= \bhatA \bhatx(t),\\
        \bcaly(t) &= \bhatF \bhatx(t),
    \end{split}
    \label{eq:reduced_LTI_system}
    \end{align}
    with reduced operators in~\eqref{eq:reduced_LTI_system} given by:
    \begin{align}
        \bhatA =\bcalV_\rho^\top \bA\bcalU_\rho \in \R^{\rho\times \rho}, && \bhatF = \bF\bcalU_\rho\in\R^{d_\text{out}\times \rho}.
        \label{eq:BTBI_reduced_operators}
    \end{align}
    The reduced operators~\eqref{eq:BTBI_reduced_operators} can be used to cheaply integrate~\eqref{eq:reduced_LTI_system}, which is now $\rho$-dimensional, thus yielding efficient approximate numerical simulations of~\eqref{eq:LTI_system}. The reduced operators in~\eqref{eq:BTBI_reduced_operators} also implicitly define reduced forward maps
    \begin{align}
        \bH_\text{BT} \equiv \begin{bmatrix}
            \bhatF e^{\bhatA t_1}\\
            \vdots \\ \bhatF e^{\bhatA t_m}
        \end{bmatrix}\bcalV_\rho^\top, && 
        \bH_\text{RLS,BT} = 
        \begin{bmatrix}
            \bH_\text{BT} \\ \bI
        \end{bmatrix},
        \label{eq:G_BT}
    \end{align}
    and corresponding posterior mean and covariance approximations,
    \begin{align}
        \GposBT = \big(\bH_\text{BT}^\top \bGamma^{-1}\bH_\text{BT} + \bGamma_\text{pr}^{-1}\big)^{-1}, && \muposBT = \GposBT\bH_\text{BT}\bGamma^{-1}\by.
        \label{eq:BT_posterior}
    \end{align}
The work in~\cite{qian_model_2022} shows that if $\Gpr$ satisfies $\bA\Gpr + \Gpr\bA^\top\preceq\bzero$, then $\Gpr$ can be interpreted as the result of spinning up the system~\eqref{eq:LTI_system} from $t=-\infty$ with a white noise input. This interpretation coincides with the system-theoretic definition of a reachability Gramian~\cite{antoulas_approximation_2005} and allows the reduced system to inherit system-theoretic error bounds and stability guarantees~\cite{qian_model_2022}. Furthermore, the work~\cite{qian_model_2022} shows that $\GposBT$ converges in certain continuous-time limits of the observation process to an optimal posterior covariance approximation for linear Gaussian Bayesian inverse problems defined in~\cite{spantini_optimal_2015}.

\subsection{Near-optimality of BT-accelerated ensemble Kalman RMLE}\label{subsec:BTEKI_result}
For linear Bayesian smoothing problems of the form described in \Cref{subsec:LTI_IP}, we propose a strategy to reduce the computational cost of our proposed ensemble Kalman RMLE algorithm by approximating the original high-dimensional system~\eqref{eq:LTI_system}, which is expensive to integrate, with the low-dimensional reduced system~\eqref{eq:reduced_LTI_system} which may be simulated more cheaply. 
We now provide a convergence result for this proposed acceleration strategy. The following corollary follows from \Cref{thm:master_convergence} when applied to the least-squares problem defined by the approximate forward operator $\bH_\text{RLS,BT}$ in~\eqref{eq:G_BT}.  We emphasize that the forward operator $\bH_\text{RLS,BT}$ would not generally be explicitly formed as part of the computation; instead, the forward operator evaluation would rely on integrating the low-dimensional reduced system~\eqref{eq:reduced_LTI_system} rather than high-dimensional original system~\eqref{eq:LTI_system}.

\begin{corollary}\label{cor:application_to_BTBI}
    As $i\to\infty$, for all $j=1,\ldots, J$, the mean field particle update~\eqref{eq:mean_field} defined with reduced operators~\eqref{eq:G_BT}  converges in the following ways:
    \begin{enumerate}[(a)]
        \item $\E\big[\bP\bvTilde_{i+1}^{(j)}\big]\to \bP\muposBT$,
        \item $\Cov\big[\bP\bvTilde_{i+1}^{(j)}\big]\to\bP\GposBT\bP^\top$.
    \end{enumerate}
\end{corollary}

\Cref{cor:application_to_BTBI} establishes that $\bP\bvTilde_{i+1}^{(j)}$ converges in distribution to $\mathcal{N}\big(\bP\muposBT,\bP\GposBT\bP^\top\big)$, which is the projection of the reduced Gaussian posterior given by \eqref{eq:BT_posterior} under $\bP$. Note that by definition, $\textsf{ran}(\bP)\subset\textsf{ran}(\bcalU_\rho)$, the balanced truncation approximation subspace. The error of the balanced truncation approximation of our ensemble Kalman RMLE approach is therefore lower-bounded by the projection error of the balanced truncation reduced model. We will demonstrate numerically in \Cref{sec:Numerics} that when this projection error is low, sampling error due to finite ensemble size dominates the error of the method. In particular, when the discrete observation process in the smoothing problem is close to its continuous-time limit, then the balanced truncation projection error is near optimal~\cite{qian_model_2022}. In such cases, the reduced cost of the balanced truncation models enable larger ensemble sizes to be computed, leading to overall reductions in error relative to the error associated with using a smaller ensemble of particles evolved with the original high-dimensional model.

\section{Linear Analysis}\label{sec:LinearAnalysis}

This section provides a complete convergence analysis of our proposed ensemble Kalman randomized maximum likelihood estimation approach for linear
$\bH\in\R^{n\times d}$. Following the analysis approach of~\cite{qian_fundamental_2024} for basic ensemble Kalman inversion, we begin with an analysis of the mean-field data misfit behavior in observation space $\R^n$ in \Cref{subsec:observation_space_analysis} that then enables us to prove results concerning the mean-field residual behavior in state space $\R^d$. These state space results were summarized in \Cref{subsec:idealized};  \Cref{subsec:stae_space_analysis} now provides more detailed analysis, intermediate results, and proofs. 


\subsection{Analysis in observation space $\R^n$}\label{subsec:observation_space_analysis}

In this section, we present linear analysis results in the observation space $\R^n$. \Cref{subsubsec:mean_field_misfit} develops the mean field data misfit iteration, \Cref{subsubsec:spectral_analysis_bcalMTilde} presents spectral analysis of the misfit iteration map, and \Cref{subsubsec:Convergence_observation_space} proves results concerning convergence of the data misfit.

\subsubsection{Mean field observation space misfit}\label{subsubsec:mean_field_misfit}

For all particles $j = 1,\ldots, J$ and iterations $i\geq 0$, we denote by $\bthetaTilde_i^{(j)} = \bH\bvTilde_i^{(j)} - \by^{(j)}$ the \textit{mean field data misfit}, which captures the deviation of the mean field iteration \eqref{eq:mean_field} under the action of $\bH$ from the perturbed data $\by^{(j)}$. From \eqref{eq:mean_field}, we have
\begin{align}
    \bthetaTilde_{i+1}^{(j)} &= \bH \bvTilde_{i+1}^{(j)} - \by^{(j)} = \bH\bMTilde_{i:0}\bv_0^{(j)} + \bH\big(\bI-\bMTilde_{i:0}\big)\bH^+\by^{(j)} - \by^{(j)}.
    \label{eq:misfit_part1}
\end{align}
Define the \textit{mean field misfit iteration map} $\bcalMTilde_i \equiv\bGamma\big(\bH\bCTilde_i\bH^\top + \bGamma\big)^{-1}$ which, via Woodbury identity, can be expressed as:
\begin{align}
    \bcalMTilde_i = \bGamma\big(\bH\bCTilde_i\bH^\top + \bGamma\big)^{-1} = \bI - \bH\bCTilde_i\big(\bI+\bH^\top\bGamma^{-1}\bH\bCTilde\big)^{-1}\bH^\top\bGamma^{-1} = \big(\bI + \bH\bCTilde_i\bH^\top\bGamma^{-1}\big)^{-1}.
    \label{eq:bcalM_alt}
\end{align}
Using the above, we can derive the identity:
\begin{align*}
    \bH\bMTilde_i = \bH(\bI-\bKTilde_i\bH) = (\bI-\bH\bKTilde_i)\bH = \big(\bI - \bH\bCTilde_i\bH^\top\big(\bH\bCTilde_i\bH^\top+\bGamma\big)^{-1}\big)\bH = \big(\bI + \bH\bCTilde_i\bH^\top\bGamma^{-1}\big)^{-1}\bH =\bcalMTilde_i\bH,
\end{align*}
where we have defined the mean field Kalman gain $\bKTilde_i = \bCTilde_i\bH^\top\big(\bH\bCTilde_i\bH^\top+\bGamma\big)^{-1}$. Repeated use of the above identity yields $\bH\bMTilde_{i:0} = \bcalMTilde_{i:0}\bH$, where $\bcalMTilde_{i:0}\equiv \bcalMTilde_i\cdots\bcalMTilde_0$, and allows us to express the misfit~\eqref{eq:misfit_part1} as:
\begin{align}
        \bthetaTilde_{i+1}^{(j)} = \bcalMTilde_{i:0}\bH\bv_0^{(j)} + \big(\bI - \bcalMTilde_{i:0}\big)\bH\bH^+\by^{(j)} - \by^{(j)} =      \bcalMTilde_{i:0}\big(\bH\bv_0^{(j)} - \by^{(j)}\big)= \bcalMTilde_{i:0}\btheta_0^{(j)},\label{eq:data_misfit_to_start}
    \end{align}
    where we expand each $\bcalMTilde_i$ using the second equality of~\eqref{eq:bcalM_alt} and use the fact that $\bH\bH^+$ is the $\bGamma^{-1}$-orthogonal projection onto $\range(\bH)$. Understanding the evolution of data misfit will therefore require an understanding of the map $\bcalMTilde_{i:0}$.

\subsubsection{Spectral analysis of mean field misfit map}\label{subsubsec:spectral_analysis_bcalMTilde}

Spectral analysis of the mean field misfit map $\bcalMTilde_{i:0}$ will reveal fundamental subspaces of the observation space $\R^n$ that remain invariant under the action of $\bcalMTilde_{i:0}$. Consider the following generalized eigenvalue problem:
        \begin{align}
            \bH\bCTilde_i\bH^\top\bw_{\ell, i} = \lambda_{\ell,i}\bGamma\bw_{\ell, i}\quad\text{for}\quad \ell = 1,\ldots, n, \quad i\geq 0,
            \label{eq:GEV_idealized}
        \end{align}
        and note that all generalized eigenvalues of~\eqref{eq:GEV_idealized} are nonnegative.

\begin{proposition}\label{prop:GEV_idealized}
             The eigenvectors of~\eqref{eq:GEV_idealized} are constant across all iterations, that is $\bw_{\ell, i+1} = \bw_{\ell, i}$ for all $i\geq 0$. The eigenvalue associated with the $\ell$-th eigenvector satisfies:
            \begin{align}
                \lambda_{\ell, i+1} = 1 - 2\prod_{k=0}^i(1+\lambda_{\ell, k})^{-1} + (1+\lambda_{\ell, 0})\prod_{k=0}^i(1+\lambda_{\ell, k})^{-2}.
                \tag{\ref{eq:idealized_eigen_evolution}}
            \end{align}
        \end{proposition}

\begin{proof}
    The proof will follow by induction on $i$. Note that~\eqref{eq:idealized_cov} implies
    \begin{align}
             \bH\bCTilde_{i+1}\bH^\top = \bcalMTilde_{i:0}\bH\bCTilde_0\bH^\top\bcalMTilde_{i:0}^{\top} + \big(\bI - \bcalMTilde_{i:0}\big)\bGamma\big(\bI - \bcalMTilde_{i:0}\big)^\top,
            \label{eq:idealized_HGH}
        \end{align}
    where we used the identities that led to~\eqref{eq:data_misfit_to_start}. Then, observe that if $(\lambda_{\ell, i},\bw_{\ell, i})$ is an eigenpair of~\eqref{eq:GEV_idealized}, we have that:
            \begin{align}
                \bcalMTilde_i^{\top}\bw_{\ell, i} = {(1+\lambda_{\ell, i})^{-1}}\bw_{\ell, i}, && \bcalMTilde_i\bGamma\bw_{\ell, i} = {(1+\lambda_{\ell, i})^{-1}}\bGamma\bw_{\ell, i}.
                \label{eq:Fact1}
            \end{align}
            Using~\eqref{eq:Fact1}, we can prove the base case:
             \begin{align*}
                \bH\bCTilde_1\bH^\top \bw_{\ell, 0} &= \bcalMTilde_0\bH\bCTilde_0\bH^\top\bcalMTilde_0^\top\bw_{\ell, 0} + \big(\bI - \bcalMTilde_0\big)\bGamma\big(\bI - \bcalMTilde_0\big)^\top\bw_{\ell, 0}\\
                &= \lambda_{\ell, 0}{(1+\lambda_{\ell, 0})^{-2}}\bGamma\bw_{\ell, 0} + \left(1 - 2{(1+\lambda_{\ell, 0})^{-1}} +{(1+\lambda_{\ell, 0})^{-2}}\right)\bGamma\bw_{\ell, 0}\\
                &= \left(1 - 2{(1+\lambda_{\ell, 0})^{-1}} + (1+\lambda_{\ell, 0}){(1+\lambda_{\ell, 0})^{-2}}\right)\bGamma\bw_{\ell, 0} = \underbrace{\lambda_{\ell, 0}{(1+\lambda_{\ell, 0}})^{-1}}_{\lambda_{\ell, 1}}\bGamma\bw_{\ell, 0},
            \end{align*}
            which directly implies $\bw_{\ell, 1} = \bw_{\ell, 0}$ and the expression~\eqref{eq:idealized_eigen_evolution} for $i = 0$. Now, assume that the $\ell$-th eigenvector is constant up to, and including, iteration $i$, i.e., $\bw_{\ell,i} = \cdots = \bw_{\ell, 0}$. Repeated application of \eqref{eq:Fact1} yields:
            \begin{align}
                \big(\bI - \bcalMTilde_{i:0}\big)^\top \bw_{\ell, i} = \left(1 - {\prod_{k=0}^i(1+\lambda_{\ell, k})^{-1}}\right)\bw_{\ell, i}.
                \label{eq:Fact4}
            \end{align}
            Then, the following also hold:
            \begin{align*}
                \bcalMTilde_{i:0}\bH\bGamma_0\bH^\top\bcalMTilde_{i:0}^\top \bw_{\ell, i} &= \lambda_{\ell, 0}{\prod_{k=0}^i(1+\lambda_{\ell, k})^{-2}}\bGamma\bw_{\ell, i},\\
                \big(\bI - \bcalMTilde_{i:0}\big)\bGamma(\bI - \bcalMTilde_{i:0})^\top\bw_{\ell, i} &=\left(1 - 2{\prod_{k=0}^i(1+\lambda_{\ell, k})^{-1}} + {\prod_{k=0}^i(1+\lambda_{\ell, k})^{-2}}\right)\bGamma\bw_{\ell, i}.
            \end{align*}
 Applying all of the above to \eqref{eq:idealized_HGH} gives
            \begin{align*}
                \bH\bCTilde_{i+1}\bH^\top\bw_{\ell, i} &=\bcalMTilde_{i:0}\bH\bCTilde_0\bH^\top\bcalMTilde_{i:0}^\top\bw_{\ell, i} + \big(\bI - \bcalMTilde_{i:0}\big)\bGamma\big(\bI - \bcalMTilde_{i:0}\big)^\top\bw_{\ell, i}\\
                &= \lambda_{\ell, 0}{\prod_{k=0}^i(1+\lambda_{\ell, k})^{-2}}\bGamma\bw_{\ell, i} + \left(1 - 2{\prod_{k=0}^i(1+\lambda_{\ell, k})^{-1}} + {\prod_{k=0}^i(1+\lambda_{\ell, k})^{-2}}\right)\bGamma\bw_{\ell, i}\\
                &= \underbrace{\left(1 - 2{\prod_{k=0}^i(1+\lambda_{\ell, k})^{-1}} +(1 + \lambda_{\ell,0}){\prod_{k=0}^i(1+\lambda_{\ell, k})^{-2}}\right)}_{\lambda_{\ell, i+1}}\bGamma\bw_{\ell, i},
            \end{align*}
            which directly implies that $\bw_{\ell, i+1} = \bw_{\ell, i}$ and completes the induction.
\end{proof}

\Cref{eq:idealized_eigen_evolution} shows that (non)zero eigenvalues remain (non)zero across all iterations. To understand the convergence behavior of eigenvalues of \eqref{eq:GEV_idealized}, we first prove some intermediate results.

\begin{lemma}\label{lem:UB}
        For all iterations $i\geq 1$, the positive eigenvalues of~\eqref{eq:GEV_idealized} are bounded above by $1$, i.e., $\lambda_{\ell, i}\leq 1$ for all $i\geq 1$, and are strictly monotone increasing, i.e., $\lambda_{\ell, 1}<\lambda_{\ell,2}< \cdots < \lambda_{\ell, i}<\cdots$, and therefore converge.
    \end{lemma}
    \begin{proof}
        Note that \eqref{eq:idealized_eigen_evolution} may be equivalently expressed as:
    \begin{align}
        \lambda_{\ell, i+1} = 1 - 2(1+\lambda_{\ell, 0})^{-1}\prod_{k=1}^i(1+\lambda_{\ell, k})^{-1} + (1+\lambda_{\ell, 0})^{-1}\prod_{k=1}^i(1+\lambda_{\ell, k})^{-2}.
        \label{eq:ideal_eig_alt}
    \end{align}
        We first show the upper bound. Using \eqref{eq:ideal_eig_alt}, we may bound each eigenvalue as follows:
        \begin{align*}
            \lambda_{\ell, i+1} &\leq 1 - 2{(1+\lambda_{\ell, 0})^{-1}\prod_{k=1}^i(1+\lambda_{\ell, k})^{-1}} + {(1+\lambda_{\ell, 0})^{-1}\prod_{k=1}^i(1+\lambda_{\ell, k})^{-1}}\\
            &= 1 - {(1+\lambda_{\ell, 0})^{-1}\prod_{k=1}^i(1+\lambda_{\ell, k})^{-1}} \leq 1.
        \end{align*}
        Next, we show monotonicity by showing that $\lambda_{\ell, i+1}-\lambda_{\ell, i}>0$. From \eqref{eq:ideal_eig_alt}, we write
        \begin{align*}
            \lambda_{\ell, i+1} - \lambda_{\ell, i} = 2{\prod_{k=0}^{i-1}(1+\lambda_{\ell, k})^{-1}} - 2{\prod_{k=0}^i(1+\lambda_{\ell, k})^{-1}} + {(1+\lambda_{\ell,0})^{-1}\prod_{k=1}^i(1+\lambda_{\ell, k})^{-2}} - {(1+\lambda_{\ell, 0})^{-1}\prod_{k=1}^{i-1}(1+\lambda_{\ell, k})^{-2}}.
        \end{align*}
        Multiplying both sides of the above expression by $\prod_{k=0}^i(1+\lambda_{\ell, k})$ gives,
        \begin{align*}
            (\lambda_{\ell,i+1} - \lambda_{\ell, i})\prod_{k=0}^i(1+\lambda_{\ell, k}) &= 2(1+\lambda_{\ell, i})-2 + {\prod_{k=1}^i(1+\lambda_{\ell, k})^{-1}} - (1+\lambda_{\ell, i}){\prod_{k=1}^{i-1}(1+\lambda_{\ell, k})^{-1}}\\
            &= {\prod_{k=0}^i(1+\lambda_{\ell, k})^{-1}} + \lambda_{\ell, i} - \lambda_{\ell, i}{\prod_{k=0}^{i-1}(1+\lambda_{\ell, k})^{-1}} +\lambda_{\ell, i} - {\prod_{k=0}^{i-1}(1+\lambda_{\ell, k})^{-1}}\\
            &= \lambda_{\ell, i} - \lambda_{\ell,i}\prod_{k=0}^{i-1}(1+\lambda_{\ell, k})^{-1} + \lambda_{\ell, i} - \lambda_{\ell, i}\prod_{k=0}^i(1+\lambda_{\ell, k})^{-1} >0.
        \end{align*}
        Since $0<\prod_{k=0}^i(1+\lambda_{\ell, k})^{-1}<1$ for all $i\geq 0$, it follows that $\lambda_{\ell,i+1}>\lambda_{\ell, i}$.
    \end{proof}

We now prove \Cref{prop:GEV_statespace_alt}, which characterizes the limit and convergence rate of the sequence~\eqref{eq:idealized_eigen_evolution}.

    \begin{proof}[Proof of \Cref{prop:limit}]
        Starting from \eqref{eq:ideal_eig_alt}, we see that we can express,
        \begin{align*}
            1 - \lambda_{\ell, i+1} = 2{(1+\lambda_{\ell, 0})^{-1}\prod_{k=1}^i(1+\lambda_{\ell, k})^{-1}} - {(1+\lambda_{\ell, 0})^{-1}\prod_{k=1}^i(1+\lambda_{\ell, k})^{-2}}\leq 2{(1+\lambda_{\ell, 0})^{-1}\prod_{k=1}^i(1+\lambda_{\ell, k})^{-1}}.
        \end{align*}
        By the monotonicity of the eigenvalue sequence (\Cref{lem:UB}), we may bound:
        \begin{align}
            2{(1+\lambda_{\ell, 0})^{-1}\prod_{k=1}^i(1+\lambda_{\ell, k})^{-1}} &\leq 2{(1+\lambda_{\ell,0})^{-1}(1+\lambda_{\ell, 1})^{-i}} = 2{(1+\lambda_{\ell, 0})^{-1}\left(1+\frac{\lambda_{\ell, 0}}{1+\lambda_{\ell, 0}}\right)^{-i}}\nonumber \\
            &= 2(1+\lambda_{\ell, 0})^{i-1}{(1+2\lambda_{\ell, 0})^{-i}} = \frac{2}{1+2\lambda_{\ell, 0}}\left(1 - \frac{\lambda_{\ell, 0}}{1 + 2\lambda_{\ell,0}}\right)^{i-1} \nonumber\\
            &\leq \frac{2}{1+2\lambda_{\ell, 0}}e^{-(i-1)\gamma_\ell},
            \label{eq:product_bound}
        \end{align}
        where $\gamma_\ell = \frac{\lambda_{\ell, 0}}{1 + 2\lambda_{\ell, 0}}$, and we made use of the fact that $\lambda_{\ell, 1} = \frac{\lambda_{\ell, 0}}{1+\lambda_{\ell, 0}}$ and the inequality $(1-x)^{\beta}\leq e^{-\beta x}$. Now, fix $\epsilon > 0$. By choosing,
        \begin{align*}
            i \geq \frac{1}{\gamma_\ell}\left[\log\left(\frac{1}{\epsilon}\right) + \log\left(\frac{2}{1+2\lambda_{\ell, 0}}\right)\right] + 1,
        \end{align*}
        we guarantee that $1 - \lambda_{\ell,i+1} <\epsilon$, which satisfies the definition of the limit of a sequence and establishes the promised result.
    \end{proof}
    \begin{corollary}\label{cor:fixed_order_of_eigenvalues}
        Let $\{(\lambda_{\ell, 1},\bw_{\ell, 1})\}_{\ell=1}^n$ denote the eigenvalue-eigenvector pairs of~\eqref{eq:GEV_idealized} at $i=1$, arranged in non-increasing order so that $\lambda_{1,1}\geq \lambda_{2,1}\geq \cdots \geq \lambda_{n,1}$. This ordering is preserved for all $i>1$. 
    \end{corollary}
    \begin{proof}
        Note that, by working with the expression for $\lambda_{\ell, i}$ in \eqref{eq:idealized_eigen_evolution}, we can express $\lambda_{\ell, i+1}$ as:
        \begin{align*}
            \lambda_{\ell, i+1} = 2\lambda_{\ell, i}(1+\lambda_{\ell, i})^{-1} - \lambda_{\ell,i}(1+\lambda_{\ell, i})^{-2}(1+\lambda_{\ell, 0})^{-1}\prod_{k=1}^{i-1} (1+\lambda_{\ell, k})^{-2},
        \end{align*}
        and note that $0<(1+\lambda_{\ell, 0})^{-1}\prod_{k=1}^{i-1}(1+\lambda_{\ell, k})^{-2}<1$. The claim follows by noticing that the map $2x(1+x)^{-1}-\alpha x(1+x)^{-2}$ is monotone increasing for all $x\in[0,1]$ and $0<\alpha<1$.
    \end{proof}

Let $\bW=[\bw_1,\ldots,\bw_r]$ denote the matrix of positive eigenvectors of~\eqref{eq:GEV_idealized}, ordered as described in \Cref{cor:fixed_order_of_eigenvalues}, noting that we have dropped the iteration index from $\bw_{\ell,i}$ because the eigenvectors are constant for all iterations. We impose the normalization $\bW^\top\bGamma\bW = \bI_r$. Then, $\bH\bCTilde_i\bH^\top\bW = \bGamma\bW\bLambda_{1:r}^{(i)}$, with $\bLambda_{1:r}^{(i)}=\textsf{diag}(\lambda_{1,i},\ldots,\lambda_{r,i})$ as before. Note that $\range(\bW)$ is an invariant subspace under the action of $\bcalMTilde_{i:0}$. We can define this subspace and its complement using spectral projectors of $\bcalMTilde_{i:0}$ as follows:

    \begin{proposition}\label{prop:SpectralMaps_bcalM}
    Let $\bcalP = \bGamma\bW\bW^\top$, and $\bcalS = \bI-\bcalP$. Then, $\bcalP$ and $\bcalS$ are spectral projectors associated with the misfit map $\bcalMTilde_{i:0}$. This means that $\bcalP^2 = \bcalP, \bcalS^2 = \bcalS$ and $\bcalP\bcalMTilde_{i:0} = \bcalMTilde_{i:0}\bcalP, \bcalS\bcalMTilde_{i:0} = \bcalMTilde_{i:0}\bcalS$, with $\bcalP\bcalS=\bcalS\bcalP=  \bzero$ and $\bcalP+\bcalS = \bI_n$.
    \end{proposition}
    \begin{proof}
        We begin by establishing that:
        \begin{align*}
            \bcalP^2 = \bGamma\bW\bW^\top\bGamma\bW\bW^\top = \bGamma\bW\bW^\top = \bcalP,
        \end{align*}
        where we used the fact that, under our normalization, $\bW^\top\bGamma\bW = \bI_r$. The results $\bcalS^2 = \bcalS$ and $\bcalP\bcalS=\bcalS\bcalP=\bzero$ follow by definition. Next, we note that we can decompose $\bcalMTilde_i = \bGamma\bW\big(\bI+\bLambda_{1:r}^{(i)}\big)^{-1}\bW^\top + \bcalS$. Indeed:
        \begin{align*}
            \bcalMTilde_i^{-1}\bcalMTilde_i &= \big(\bI + \bH\bCTilde_i\bH^\top\bGamma^{-1}\big)\big(\bGamma\bW\big(\bI+\bLambda_{1:r}^{(i)}\big)^{-1}\bW^\top + \bcalS\big)\\
            &= \bGamma\bW\big[\bI + \bLambda_{1:r}^{(i)}\big]\big(\bI+\bLambda_{1:r}^{(i)}\big)^{-1}\bW^\top + \bcalS + \bH\bCTilde_i\bH^\top\bGamma^{-1}\bcalS\\
            &= \bGamma\bW\bW^\top + \bcalS = \bcalP+\bcalS = \bI_n.
        \end{align*}
        Using the above decomposition, we can evaluate:
        \begin{align*}
           \bcalP\bcalMTilde_i &= \bcalP\big(\bGamma\bW\big(\bI+\bLambda_{1:r}^{(i)}\big)^{-1}\bW^\top + \bcalS\big) = \bGamma\bW\big(\bI+\bLambda_{1:r}^{(i)}\big)^{-1}\bW^\top = \bcalMTilde_i\bcalP.
        \end{align*}
        Successive applications of the above result give the claim $\bcalP\bcalMTilde_{i:0} = \bcalMTilde_{i:0}\bcalP$, with the analogous results for $\bcalS$ following directly.
    \end{proof}

The subspace $\textsf{ran}(\bcalP)=\textsf{ran}(\bGamma\bW)$ is associated with the eigenvectors of~\eqref{eq:GEV_idealized} that have nonzero eigenvalues. These directions are associated with directions that are observable under $\bH$ and populated by the ensemble in the sense that $\textsf{ran}(\bGamma\bW)=\textsf{ran}(\bH\bCTilde_i)$. We will show that the data misfit converges in this observable and populated space, and that the misfit fails to converge in the complementary subspace $\textsf{ran}(\bcalS)$.
\subsubsection{Misfit convergence analysis in observation space}\label{subsubsec:Convergence_observation_space}

\Cref{prop:SpectralMaps_bcalM} shows that the mean field misfit $\bthetaTilde_{i+1}^{(j)}$ \eqref{eq:data_misfit_to_start} can be decomposed as:
    \begin{align}
        \bthetaTilde_{i+1}^{(j)} &= \bcalP\bthetaTilde_{i+1}^{(j)} + \bcalS\bthetaTilde_{i+1}^{(j)} = \bcalP\bcalMTilde_{i:0}\btheta_0^{(j)} + \bcalS\bcalMTilde_{i:0}\btheta_0^{(j)}=\bcalMTilde_{i:0}\bcalP\btheta_0^{(j)} + \bcalMTilde_{i:0}\bcalS\btheta_0^{(j)}.
        \label{eq:misfit_spectral_decomposition}
    \end{align}
    In order to understand how $\bthetaTilde_{i+1}^{(j)}$ evolves as $i\to\infty$, we need to understand how the misfit evolves in each invariant subspace, $\textsf{ran}(\bcalP)$ and $\textsf{ran}(\bcalS)$. We first present a lemma that shows how we can construct $\bcalMTilde_{i:0}$ using eigenvectors and eigenvalues of \eqref{eq:GEV_idealized}.

\begin{lemma}\label{lem:Mcal_compound_spectral}
        We have that:
        \begin{align}
            \bcalMTilde_{i:0} = \sum_{\ell=1}^r \left(\prod_{k=0}^i\frac{1}{1+\lambda_{\ell, k}}\right)\bGamma\bw_\ell\bw_\ell^\top  + \bcalS.
            \label{eq:Mcal_compound_spectral}
        \end{align}
    \end{lemma}
    \begin{proof}
        Recall that $\bcalMTilde_i = \bGamma\bW\big(\bI+\bLambda_{1:r}^{(i)}\big)^{-1}\bW^\top + \bcalS$. Hence,
        \begin{align*}
            \bcalMTilde_i = \sum_{\ell=1}^r \frac{1}{1+\lambda_{\ell, i}}\bGamma\bw_\ell\bw_\ell^\top + \bcalS.
        \end{align*}
        \Cref{eq:Mcal_compound_spectral} follows by using the above expression for each $\bcalMTilde_i$ combined with the properties of $\bcalS$ and the $\bGamma$-orthogonality of the eigenvectors.
    \end{proof}

We are now ready to state our main result regarding the convergence of the mean field misfit.
    \begin{theorem}\label{thm:data_misfit_convergence}
        For all particles $j=1,\ldots, J$, the following hold for all $i\geq 0$:
        \begin{enumerate}[(a)]
            \item $\bcalS\bthetaTilde_{i+1}^{(j)} = \bcalS\btheta_0^{(j)}$,
            \item $\bcalP\bthetaTilde_{i+1}^{(j)}\leq e^{-(i-1)\gamma}\bcalP\btheta_0^{(j)}$,
        \end{enumerate}
        where $\gamma=\min_{1\leq \ell\leq r} \gamma_\ell$, and $\gamma_\ell =  \lambda_{\ell,0}/(1+2\lambda_{\ell, 0})$.
    \end{theorem}
    \begin{proof}
        From \eqref{eq:misfit_spectral_decomposition}, we have 
        \begin{align*}
            \bcalS\bthetaTilde_{i+1}^{(j)} &= \bcalMTilde_{i:0}\bcalS\bcalP\btheta_0^{(j)} + \bcalMTilde_{i:0}\bcalS^2 \btheta_0^{(j)} = \bcalMTilde_{i:0}\bcalS\btheta_0^{(j)} = \bcalS\btheta_0^{(j)},
        \end{align*}
        which follows from \Cref{lem:Mcal_compound_spectral} and gives the result in (a). To show (b), note that,
        \begin{align*}
            \bcalP\bthetaTilde_{i+1}^{(j)} &= \bcalMTilde_{i:0}\bcalP\btheta_0^{(j)} = \sum_{\ell=1}^r \left(\prod_{k=0}^i\frac{1}{1+\lambda_{\ell, k}}\right)\bGamma\bw_{\ell}\bw_{\ell}^\top\bcalP\btheta_0^{(j)},
        \end{align*}
        where we applied \Cref{lem:Mcal_compound_spectral}, and used the $\bGamma$-orthogonality of the eigenvectors. From \cref{eq:product_bound}, it follows that $\prod_{k=0}^i(1+\lambda_{\ell,k})^{-1}\leq e^{-(i-1)\gamma_\ell} \leq e^{-(i-1)\gamma}$. Using this, we may bound:
        \begin{align*}
            \bcalP\bthetaTilde_{i+1}^{(j)} &\leq e^{-(i-1)\gamma}\sum_{\ell = 1}^r \bGamma\bw_{\ell}\bw_{\ell}^\top\bcalP\btheta_0^{(j)} = e^{-(i-1)\gamma}\bcalP^2\btheta_0^{(j)} = e^{-(i-1)\gamma}\bcalP\btheta_0^{(j)},
        \end{align*}
as desired, where we used the fact that $\bcalP^2=\bcalP$.
\end{proof}

\Cref{thm:data_misfit_convergence} establishes that $\bcalP\bthetaTilde_{i+1}^{(j)}\to\bzero$ as $i\to\infty$ for all $j$. Moreover, \Cref{thm:data_misfit_convergence} directly implies the existence of a limit for $\bH\bvTilde_{i+1}^{(j)}\equiv\bhTilde_{i+1}^{(j)}$ for each $j$. In order to characterize this limit, we need an intermediate result.

\begin{lemma}\label{lem:limit_of_calM}
        We have that:
        \begin{align}
             \lim_{i\to\infty}\bcalMTilde_{i:0} = \bcalS, && \lim_{i\to\infty} \bI-\bcalMTilde_{i:0} = \bcalP.
            \label{eq:limit_of_calM}
        \end{align}
    \end{lemma}
    \begin{proof}
        For the first result, we use \Cref{lem:Mcal_compound_spectral} to conclude:
        \begin{align*}
            \lim_{i\to\infty}\bcalMTilde_{i:0} &= \sum_{\ell=1}^r\left(\lim_{i\to\infty}\prod_{k=0}^i \frac{1}{1+\lambda_{\ell, k}}\right)\bGamma\bw_\ell\bw_\ell^\top+\bcalS \leq \sum_{\ell=1}^r \lim_{i\to\infty}e^{-(i-1)\gamma}\bGamma\bw_\ell\bw_\ell^\top + \bcalS = \bcalS,
        \end{align*}
        where $\gamma$ is defined as in \Cref{thm:data_misfit_convergence}. The second result follows by definition of $\bcalS=\bI-\bcalP$.
    \end{proof}
    We are now ready to state our main result.

\begin{theorem}\label{thm:master_convergence_observation_space}
    As $i\to\infty$, we have that for all $j=1,\ldots, J$:
    \begin{enumerate}[(a)]
        \item $\bhTilde_{i+1}^{(j)}\to \bhTilde_\infty^{(j)} = \bcalP\by^{(j)} + \bcalS\bh_0^{(j)}$,
        \item $\E\big[\bhTilde_{i+1}^{(j)}\big]\to \E\big[\bhTilde_\infty^{(j)}\big] = \bcalP\by +\bcalS\bH\bmu_0$,
        \item $\Cov\big[\bhTilde_{i+1}^{(j)}\big]\to\Cov\big[\bhTilde_\infty^{(j)}\big] = \bcalP\bGamma\bcalP^\top + \bcalS\bH\bCTilde_0\bH^\top\bcalS^\top$.
    \end{enumerate}
\end{theorem}

\begin{proof}
    Statement (a) is a formal consequence of \Cref{thm:data_misfit_convergence}. For (b), we compute:
    \begin{align*}
        \E\big[\bhTilde_{i+1}^{(j)}\big] &= \E\big[\bH\bvTilde_{i+1}^{(j)}\big] = \E\big[\bcalMTilde_{i:0}\bH\bv_0^{(j)} + \big(\bI-\bcalMTilde_{i:0}\big)\by^{(j)}\big] = \big(\bI - \bcalMTilde_{i:0}\big)\by + \bcalMTilde_{i:0}\bH\bmu_0,
    \end{align*}
    and apply \Cref{lem:limit_of_calM}. Finally, for (c), note that $\Cov\big[\bhTilde_{i+1}^{(j)}\big] = \bH\bCTilde_{i+1}\bH^\top$ and apply \Cref{lem:limit_of_calM} to \cref{eq:idealized_HGH}.
\end{proof} 

\subsection{Analysis in state space $\R^d$}\label{subsec:stae_space_analysis}
We now build on our results in the observation space $\R^n$ to prove analogous results in the state space $\R^d$. These results were summarized in \Cref{subsec:idealized}. In this section, we provide more details and proofs. \Cref{subsubsec:decomposition_of_state_space} presents detailed spectral analysis of the mean field residual iteration map and \Cref{subsubsec:state_convergence} completes our state space convergence theory.

\subsubsection{Spectral analysis of mean field residual map}\label{subsubsec:decomposition_of_state_space}

We begin by showing that the eigenvalues of~\eqref{eq:statespace_GEV_alt} are the same as the eigenvalues of~\eqref{eq:GEV_idealized}, and that the eigenvectors of~\eqref{eq:statespace_GEV_alt} can be given in terms of the eigenvectors of~\eqref{eq:GEV_idealized}.
\begin{proof}[Proof of \Cref{prop:GEV_statespace_alt}]
            Let $\{(\lambda_{\ell, i}, \bw_{\ell, i})\}_{\ell = 1}^n$ be eigenpairs of \eqref{eq:GEV_idealized} ordered as described in~\Cref{cor:fixed_order_of_eigenvalues}. For $\ell\leq r$, let $\bu_{\ell,i} = {\lambda_{\ell, i}}^{-1}\bCTilde_i\bH^\top\bw_{\ell, i}$. Then, for all $\ell\leq r$, $\bu_{\ell, i}$ is an eigenvector of \eqref{eq:statespace_GEV_alt} with corresponding eigenvalue $\lambda_{\ell,i}$. The relationship can be verified by substitution. We now show, via induction on $i$, that eigenvectors $\bu_{\ell, i}$ are constant across iterations. First, note that using $\bMTilde_i = \big(\bI + \bCTilde_i\bH^\top\bGamma^{-1}\bH\big)^{-1}$ it is possible to show that $\bMTilde_i\bu_{\ell, i} = (1+\lambda_{\ell, i})^{-1}\bu_{\ell, i}$. Then, from~\eqref{eq:idealized_cov}, we have
            \begin{align*}
                \bu_{\ell,1} &= \lambda_{\ell,1}^{-1}\bCTilde_1\bH^\top\bw_{\ell, 1} = \lambda_{\ell, 1}^{-1}\bCTilde_1\bH^\top\bw_{\ell, 0} \\
                &= \lambda_{\ell,1}^{-1}\bMTilde_0\bCTilde_0\bH^\top\bcalMTilde_0^\top\bw_{\ell, 0} + \lambda_{\ell,1}^{-1}(\bI-\bMTilde_0)\big(\bH^\top\bGamma^{-1}\bH\big)^\dagger\bH^\top (\bI-\bcalMTilde_0)^\top\bw_{\ell, 0}\\
                &= \lambda_{\ell, 1}^{-1}\lambda_{\ell, 0}(1+\lambda_{\ell, 0})^{-2}\bu_{\ell, 0} + \lambda_{\ell, 1}^{-1}\big(1-(1+\lambda_{\ell, 0})\big)^{-1}(\bI-\bcalM_0)\big(\bH^\top\bGamma^{-1}\bH\big)^\dagger\bH^\top\bw_{\ell, 0}.
            \end{align*}
            Using the above, we may evaluate the base case:
            \begin{align*}
                \bCTilde_0\bH^\top\bGamma^{-1}\bH\bu_{\ell, 1} &= \bCTilde_0\bH^\top\bGamma^{-1}\bH\big(\lambda_{\ell, 1}^{-1}\lambda_{\ell, 0}(1+\lambda_{\ell, 0})^{-2}\bu_{\ell, 0} + \lambda_{\ell, 1}^{-1}\big(1-(1+\lambda_{\ell, 0})\big)^{-1}(\bI-\bcalM_0)\big(\bH^\top\bGamma^{-1}\bH\big)^\dagger\bH^\top\bw_{\ell, 0}\big)\\
                &= \lambda_{\ell,1}^{-1}\big(\lambda_{\ell, 0}^2(1+\lambda_{\ell, 0})^{-2} + \lambda_{\ell, 0}\big(1-(1+\lambda_{\ell, 0})^{-1}\big)-\lambda_{\ell,0}(1+\lambda_{\ell, 0})^{-1}\big(1-(1+\lambda_{\ell, 0})\big)^{-1}\big)\bu_{\ell,0}\\
                &= \lambda_{\ell, 0}\bu_{\ell, 0}.
            \end{align*}
            The base case is confirmed. Now, assume that $\bu_{\ell,i}=\cdots=\bu_{\ell, 0}$. The inductive hypothesis can be used to show that:
            \begin{align*}
                \bu_{\ell, i+1} &= \lambda_{\ell, i+1}^{-1}\bCTilde_{i+1}\bH^\top\bw_{\ell, i+1} = \lambda_{\ell, i+1}^{-1}\bCTilde_{i+1}\bH^\top\bw_{\ell, i}\\
                &=\lambda_{\ell, i+1}^{-1}\left(\prod_{k=0}^i (1+\lambda_{\ell, k})^{-1}\right)\bMTilde_{i:0}\bCTilde_0\bH^\top\bw_{\ell, i} + \lambda_{\ell, i+1}^{-1}\left(1-\prod_{k=0}^i (1+\lambda_{\ell, k})^{-1}\right)(\bI-\bMTilde_{i:0})\big(\bH^\top\bGamma^{-1}\bH\big)^\dagger\bH\bw_{\ell, i} \\
                &= \lambda_{\ell,i+1}^{-1}\lambda_{\ell, 0}\left(\prod_{k=0}^i(1+\lambda_{\ell, k})^{-2}\right) + \lambda_{\ell, i+1}^{-1}\left(1-\prod_{k=0}^i (1+\lambda_{\ell, k})^{-1}\right)(\bI-\bMTilde_{i:0})\big(\bH^\top\bGamma^{-1}\bH\big)^\dagger\bH\bw_{\ell, i},
            \end{align*}
            which in turn can be used to show that:
            \begin{align*}
                \bCTilde_i\bH^\top\bGamma^{-1}\bH\bu_{\ell,i+1} &= \left(\lambda_{\ell, i+1}^{-1}\lambda_{\ell,i}\lambda_{\ell, 0}\left(\prod_{k=0}^i (1+\lambda_{\ell, k})^{-2}\right) + \lambda_{\ell,i+1}^{-1}\lambda_{\ell, i}\left(1-\prod_{k=0}^i (1+\lambda_{\ell, k})\right)^2\right)\bu_{\ell,i}.
            \end{align*}
            This completes the induction and shows that eigenvectors of~\eqref{eq:statespace_GEV_alt} are constant across all iterations.

        \end{proof}

\Cref{prop:bbM_spectral_projectors} uses the eigenvectors of~\eqref{eq:statespace_GEV_alt} to define complementary spectral projectors $\bP$ and $\bS$ of state space, similar to \Cref{prop:SpectralMaps_bcalM} in observation space.

        \begin{proof}[Proof of \Cref{prop:bbM_spectral_projectors}]
            We begin by verifying:
            \begin{align*}
                \bP^2 &= \bU\bU^\top\bH^\top\bGamma^{-1}\bH\bU\bU^\top\bH^\top\bGamma^{-1}\bH = \bU\bI_r\bU^\top\bH^\top\bGamma^{-1}\bH = \bP.
            \end{align*}
            The claim $\bS^2=\bS$ follows by definition of $\bS$ and by $\bP^2=\bP$. Using the definitions, we can also verify $\bP\bS = \bS\bP =\bzero$. Note that we can express $\bMTilde_i = \bU\big(\bLambda_{1:r}^{(i)} + \bI\big)^{-1}\bU^\top\bH^\top\bGamma^{-1}\bH + \bS$. Indeed,
            \begin{align*}
                \bMTilde_i^{-1}\bMTilde_i &= \big(\bI + \bCTilde_i\bH^\top\bGamma^{-1}\bH\big)\big(\bU\big(\bLambda_{1:r}^{(i)} + \bI\big)^{-1}\bU^\top\bH^\top\bGamma^{-1}\bH  + \bS\big)\\
                &= \bU\big[\bI + \bLambda_{1:r}^{(i)}\big]\big(\bLambda_{1:r}^{(i)} + \bI\big)^{-1}\bU^\top\bH^\top\bGamma^{-1}\bH + \bS\\
                &= \bU\bU^\top \bH^\top\bGamma^{-1}\bH +\bS = \bP + \bS  =\bI.
            \end{align*}
            Using the above, we may verify that
            \begin{align*}
                \bP\bMTilde_i &= \bU\bU^\top \bH^\top\bGamma^{-1}\bH \big(\bU\big(\bLambda_{1:r}^{(i)} + \bI\big)^{-1}\bU^\top\bH^\top\bGamma^{-1}\bH + \bS\big)\\
                &= \bU\big(\bLambda_{1:r}^{(i)} + \bI\big)^{-1}\bU^\top\bH^\top\bGamma^{-1}\bH\\
                &= \big(\bU\big(\bLambda_{1:r}^{(i)} + \bI\big)^{-1}\bU^\top\bH^\top\bGamma^{-1}\bH + \bS\big)\bU\bU^\top \bH^\top\bGamma^{-1}\bH = \bMTilde_i\bP.
            \end{align*}
            Successive application of the above result gives $\bP\bMTilde_{i:0} = \bMTilde_{i:0}\bP$ and the results for $\bS$ follow by definition.
        \end{proof}

\subsubsection{Residual convergence analysis in state space}\label{subsubsec:state_convergence}
\Cref{thm:state_misfit_convergence} states that the mean field residual $\bomega_{i+1}^{(j)}$ converges to zero in $\textsf{ran}(\bcalP)$ and remains constant in $\textsf{ran}(\bcalS)$. 
    
    \begin{proof}[Proof of \Cref{thm:state_misfit_convergence}]
        Statement (a) follows the same logic as in the proof of \Cref{thm:data_misfit_convergence}. To arrive at (b), we note that:
        \begin{align*}
            \bMTilde_{i:0} &= \sum_{\ell = 1}^r\left(\prod_{k=0}^i \frac{1}{1+\lambda_{\ell, i}}\right)\bu_\ell\bu_\ell^\top \bH^\top\bGamma^{-1}\bH  + \bS,
        \end{align*}
        which follows by noticing that,
        \begin{align*}
            \bMTilde_i = \sum_{\ell = 1}^r \frac{1}{1+\lambda_{\ell, i}}\bu_\ell\bu_\ell^\top\bH^\top\bGamma^{-1}\bH + \bS,
        \end{align*}
        and by applying the weighted normalization of the eigenvectors. Using the above, we express
        \begin{align*}
            \bP\bomegaTilde_{i+1}^{(j)} &= \bMTilde_{i:0}\bP\bomega_0^{(j)}= \sum_{\ell = 1}^r \left(\prod_{k=0}^i \frac{1}{1+\lambda_{\ell, k}}\right)\bu_\ell\bu_\ell^\top \bH^\top\bGamma^{-1}\bH\bP\bomega_0^{(j)},
        \end{align*}
        and we follow the same steps as in the proof of \Cref{thm:data_misfit_convergence}.
        
    \end{proof}

\Cref{thm:master_convergence} uses the results of \Cref{thm:state_misfit_convergence} to establish a limit for each mean field particle, and subsequent limits for the ensemble mean and covariance. 

\begin{proof}[Proof of \Cref{thm:master_convergence}]
   \Cref{thm:state_misfit_convergence} shows that $\bP\bomegaTilde_{i+1}^{(j)}\to\bzero$ as $i\to\infty$; this gives the result presented in (a). A direct analog of \Cref{lem:limit_of_calM} shows that:
   \begin{align*}
           \lim_{i\to\infty}\bMTilde_{i:0} = \bS, && \lim_{i\to\infty} \bI-\bMTilde_{i:0} = \bP.
    \end{align*}    
   For (b), take the expectation (with respect to the particles and the perturbations) of \eqref{eq:mean_field} and apply the above limits. Finally, for (c), we use \eqref{eq:idealized_cov} and apply the above limits.
\end{proof}

The last result of this analysis, \Cref{cor:convergence_in_distribution}, shows convergence in distribution.

\begin{proof}[Proof of \Cref{cor:convergence_in_distribution}]
    First, note that:
    \begin{align*}
        \bP\bvTilde_\infty^{(j)} = \bP\bH^+\by + \bP\bH^+\beps^{(j)} = \bP\bv^\star + \bP\bH^+\bGamma^{\frac{1}{2}}\bzeta^{(j)},
    \end{align*}
    where $\bzeta^{(j)}\sim\mathcal{N}(\bzero,\bI)$. By definition, the above is a multivariate normal random vector with mean $ \bP\bv^\star$ and covariance $\big(\bP \bH^+\bGamma^\frac{1}{2}\big)\big(\bP \bH^+\bGamma^\frac{1}{2}\big)^\top = \bP\big(\bH^\top\bGamma^{-1}\bH\big)^\dagger\bP^\top$.
\end{proof}

\section{Numerical results}\label{sec:Numerics}

This section presents numerical results that demonstrate our method and illustrate its theoretical properties. \Cref{subsec:limits_numerical} presents results illustrating both the state and observation space convergence properties of the method. \Cref{subsec:Bayesian_numerical} demonstrates that our method generates samples from the posterior distribution of a linear Gaussian smoothing problem of the form described in \Cref{subsec:LTI_IP}, and demonstrates that the model reduction method described in \Cref{subsec:BTBI} can be used to reduce the computational cost of the method with negligible loss in accuracy. Julia code that produces all numerical results can be found at \url{https://github.com/pstavrin/EKRMLE}.

\subsection{Numerical illustration of convergence theory}\label{subsec:limits_numerical}
We follow the example of \cite{qian_fundamental_2024} to define an inverse problem of the form \eqref{eq:EKI_min} as follows. We have $n=500$ observations and $d=1000$ states with randomly generated $\bH,\bGamma,\by$ and initial ensemble $\bv_0^{(1:J)}$, such that: $\bGamma\in\R^{n\times n}$ is a random symmetric positive definite and $\bH\in\R^{n\times d}$ is constructed so that both $\kernel(\bH)$ and $\kernel\big(\bH^\top\big)$ are non-trivial. Observations $\by$ are created by applying $\bH$ to a random vector in $\R^d$ and adding noise drawn from $\mathcal{N}(\bzero,\bGamma)$. The initial ensemble is randomly drawn such that particles typically have nonzero components in both $\range\big(\bH^\top\big)$ and $\kernel(\bH)$, but do \textit{not} contain $\range\big(\bH^\top\big)$ in their span. Note that if the dimensions of the problem are different, particularly if $n\geq d$, the exact manifestations of the convergent and non-convergent subspaces will change, but our convergence results still apply.

We quantify convergence in the convergent and non-convergent subspaces by considering the relative ensemble errors and the relative covariance errors:
\begin{subequations}
    \begin{align}
    &\EE\left[\frac{\norm{\bcalP\big(\bh_i^{(1:J)} - \bhTilde_\infty^{(1:J)}\big)}}{\norm{\bcalP\bhTilde_\infty^{(1:J)}}}\right], && \EE\left[\frac{\norm{\bP\big(\bv_i^{(1:J)} -\bvTilde_\infty^{(1:J)}\big)}}{\norm{\bP\bvTilde_\infty^{(1:J)}}}\right],\label{eq:mean_metrics}\\
    &\frac{\norm{\CovE\big[\bcalP\big(\bh_i^{(1:J)} - \bhTilde_\infty^{(1:J)}\big)\big]}}{\norm{\CovE\big[\bcalP\bhTilde_\infty^{(1:J)}\big]}}, && \frac{\norm{\CovE\big[\bP\big(\bv_i^{(1:J)} -\bvTilde_\infty^{(1:J)}\big)\big]}}{\norm{\CovE\big[\bP\bvTilde_\infty^{(1:J)}\big]}},
    \label{cov_metrics}
\end{align}
\end{subequations}
with analogous relative error metrics for the $\bcalS/\bS$ projectors. Recall that $\bhTilde_\infty^{(j)}$ and $\bvTilde_\infty^{(j)}$ are the infinite iteration limits under the mean field iteration, as identified in \Cref{thm:master_convergence_observation_space,thm:master_convergence}, respectively. In this way, we compare the true algorithmic iterations with their mean field limits, for which we prove convergence results. We present results for a large ensemble of size $J=10000$ and for a small ensemble of size $J=10$, both for $i$ up to $i_{\rm max}=100$.

\Cref{fig:misfit_means} plots the ensemble mean errors~\eqref{eq:mean_metrics} over iterations in the convergent subspace $\textsf{ran}(\bP)$ and the non-convergent subspace $\textsf{ran}(\bS)$. Solid purple lines indicate convergence of the ensemble means in the $\bcalP/\bP$ spaces where, for a large ensemble, we observe exponential convergence of both the observation and state spaces. For a small ensemble, convergence is not achieved as the true iterations \eqref{eq:EKIRMLE_update_step} are far from the mean field setting \eqref{eq:mean_field} for which we have proved convergence. Dashed orange lines demonstrate behavior in the $\bcalS/\bS$ spaces, where we observe that the components in this subspace remain constant across all iterations.

In \Cref{fig:misfit_covs}, we plot the ensemble covariance errors~\eqref{cov_metrics} over iterations, where we observe similar trends to those in  \Cref{fig:misfit_means}. Namely, we observe exponential convergence to the limiting mean field sample covariance in the $\bcalP/\bP$ spaces and constant components in the $\bcalS/\bS$ spaces. Overall, results presented in \Cref{fig:misfit_means,fig:misfit_covs} verify the theoretical results developed in \Cref{thm:master_convergence,thm:master_convergence_observation_space}.

\begin{figure}[htb!]
    \centering
    \includegraphics[width=0.75\linewidth]{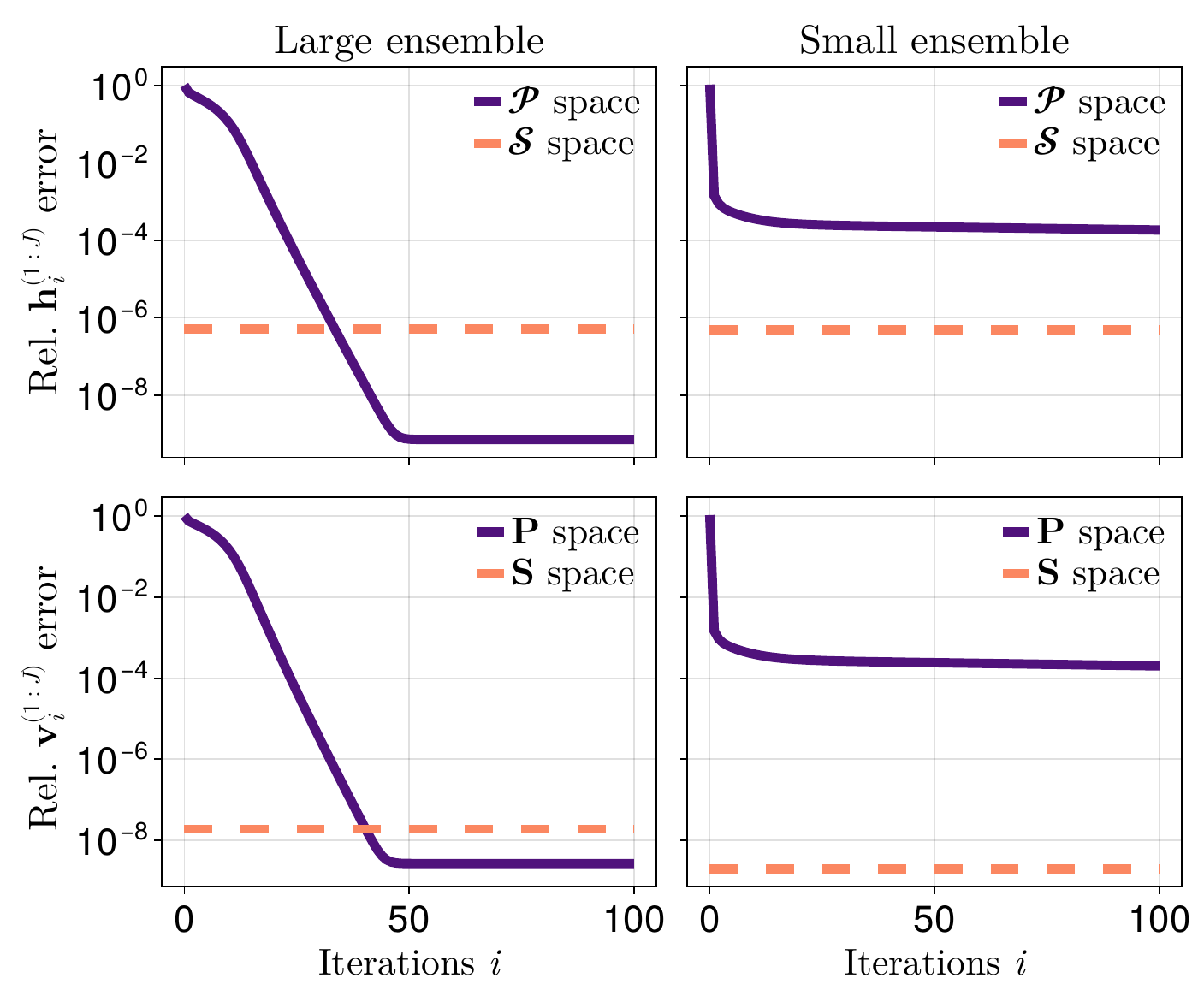}
    \caption{Mean relative errors~\eqref{eq:mean_metrics} of particles with respect to their infinite iteration limits under the mean field iteration, for a randomly generated problem with dimensions $n=500$ and $d=1000$. The large ensemble consists of $J=10000$ particles, whereas the small ensemble consists of $J=10$ particles.}
    \label{fig:misfit_means}
\end{figure}

\begin{figure}[htb!]
    \centering
    \includegraphics[width=0.75\linewidth]{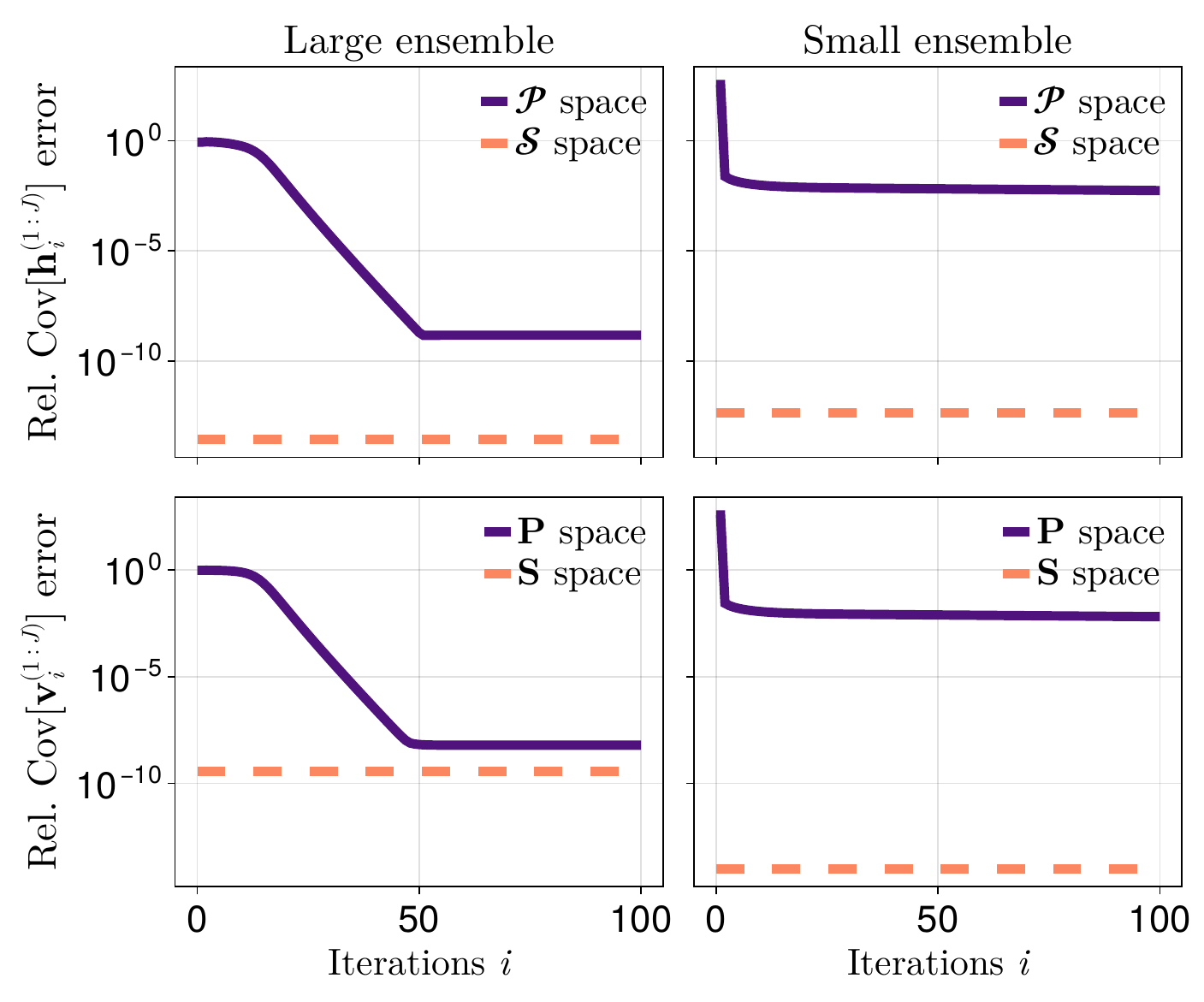}
    \caption{Relative ensemble covariance errors~\eqref{cov_metrics} with respect to the sample covariance of infinite iteration limits under the mean field iteration, for a randomly generated problem with dimensions $n=500$ and $d=1000$. The large ensemble consists of $J=10000$ particles, whereas the small ensemble consists of $J=10$ particles.}
    \label{fig:misfit_covs}
\end{figure}

\subsection{Bayesian inversion \& model reduction}\label{subsec:Bayesian_numerical}

In this section, we apply our new method (\Cref{subsec:EKRMLE}) to the Bayesian inference of the unknown initial condition of a linear dynamical system, as described in \Cref{subsec:LTI_IP}, and show that the ensemble captures posterior statistics. We additionally demonstrate that the reduced models described in \Cref{subsec:BTBI} can accelerate the algorithm with minimal loss in accuracy. We consider the heat equation benchmark problem from \cite{qian_model_2022} on a one-dimensional rod, with homogeneous Dirichlet boundary conditions, with operators defining~\eqref{eq:LTI_system} (\texttt{heat-cont.mat}) from \url{https://www.slicot.org/20-site/126-benchmark-examples-for-model-reduction}. The state and output dimensions are $d=200$ and $d_\text{out} = 1$, respectively, where the single output is the temperature measured at $2/3$ of the rod length. We define a Gaussian prior $\mathcal{N}(\bzero,\Gpr)$ where $\Gpr$ is the solution of the Lyapunov equation $\bA\Gpr + \Gpr\bA + \bI = \bzero$. 
This allows the theoretical guarantees of~\cite{qian_model_2022} to apply, so that the posterior covariance associated with the balanced truncation reduced forward model is near-optimal.

We simulate the system via forward Euler on the time interval $[0,10]$ with timestep $10^{-3}$, and we take noisy measurements at times $t=\{0.1,0.2,\ldots,10\}$, yielding a total of $n=m=100$ measurements. The noise values $\bheta$ at each measurement are i.i.d. and drawn from $\mathcal{N}(0,\sigma_\text{obs}^2)$ where we set $\sigma_\text{obs} = 0.008$, which is approximately $10\%$ of the maximum noiseless output.

We use our new method (\Cref{alg:our_method}) to solve the inverse problem formulated in \Cref{subsec:LTI_IP} and compute the full posterior statistics using the high-dimensional model as well as the reduced approximations using the reduced models described in~\eqref{eq:BT_posterior}. We consider reduced basis sizes $r=\{3,5,10,20\}$ and ensemble sizes $\{10^3, 10^4, 10^5, 10^6\}$. Due to the inherent randomness in our algorithm, we conduct $30$ Monte Carlo (MC) replicates, using the same data but a different initial ensemble, of each experiment and report the following mean relative error metrics for the ensemble mean and covariance averaged over the replicates:
\begin{align}
    \frac{\norm{\mupos - \EE\big[\bv_i^{(1:J)}\big]}_{\Gpos^{-1}}}{\norm{\mupos}_{\Gpos^{-1}}}, && \frac{\norm{\Gpos - \CovE\big[\bv_i^{(1:J)}\big]}}{\norm{\Gpos}}.
    \label{eq:error_metrics}
\end{align}

\begin{figure}[htb!]
    \centering
    \includegraphics[width=\linewidth]{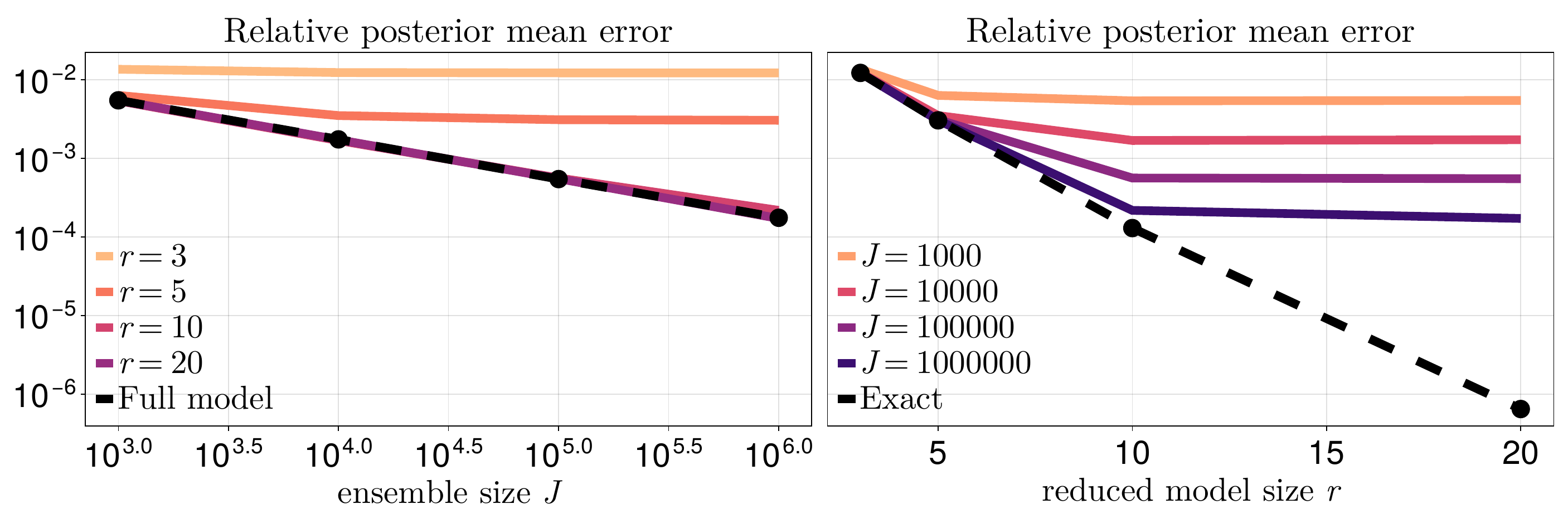}
    \caption{Relative error in posterior mean estimated via ensemble Kalman RMLE against ensemble size $J$ (left) and reduced model basis size $r$ (right). The full model line on the left panel refers to algorithmic evolution of the full model and the exact line on the right refers to errors computed with the exact posterior statistics of the reduced model \eqref{eq:BT_posterior}.}
    \label{fig:errsVSbasis}
\end{figure}

The left panel of \Cref{fig:errsVSbasis} plots error in the posterior mean against ensemble size for the reduced basis sizes we consider. The black dashed line, which represents the error of the full model, shows that the ensemble mean converges to the Bayesian posterior mean, as stated in \Cref{cor:application_to_BIP}, with accuracy improving as ensemble size increases. Note that the reduced models with basis sizes $r=10$ and $r=20$ have errors close to that of the full model; for these basis sizes the projection error is low and sampling error dominates. In contrast, for basis sizes $r =3$ and $r=5$, projection error is high and dominates sampling error. 

The right panel of \Cref{fig:errsVSbasis} plots error in the posterior mean vs reduced model sizes for different ensemble sizes. The black dashed line plots errors computed using the exact reduced posterior mean \eqref{eq:BT_posterior}; this coincides with the mean-field (infinite particle) limit. Colored lines plot the results from different finite ensemble sizes. These lines flatten out as $r$ increases, demonstrating that sampling error is dominating projection error from the balanced truncation reduced model.

\begin{figure}[htb!]
    \centering
    \includegraphics[width=\linewidth]{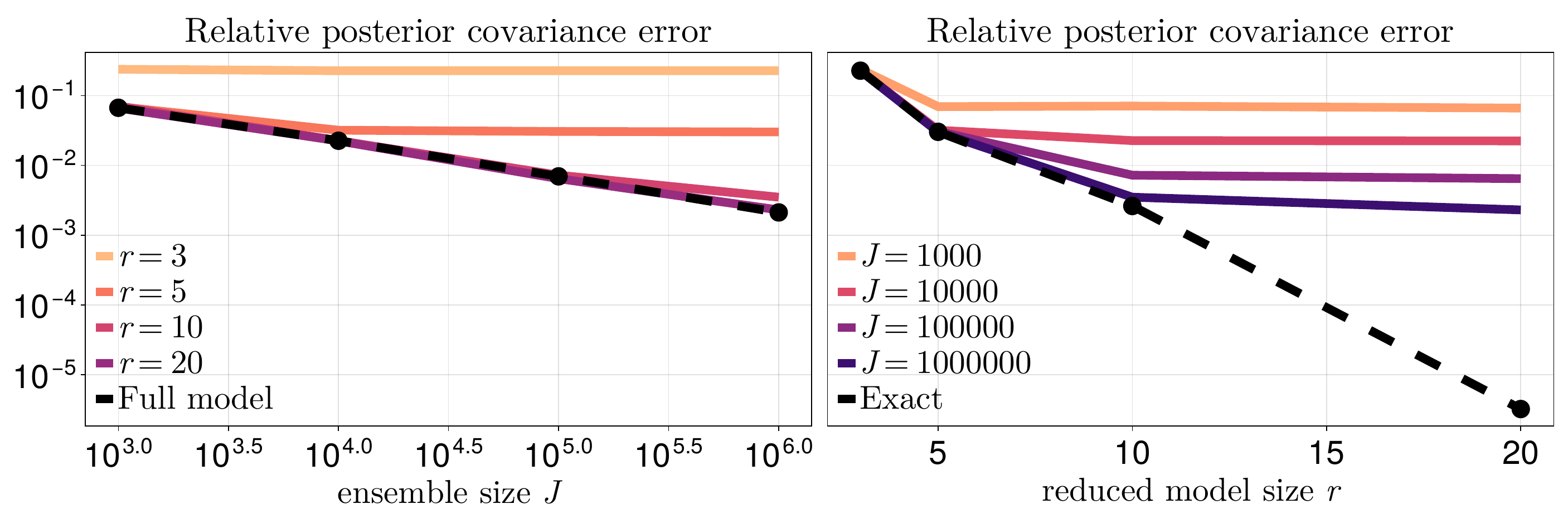}
    \caption{Relative error in posterior covariance estimated via ensemble Kalman RMLE against ensemble size $J$ (left) and reduced model basis size $r$ (right). The full model line on the left panel refers to algorithmic evolution of the full model and the exact line on the right refers to errors computed with the exact posterior statistics of the reduced model \eqref{eq:BT_posterior}.}
    \label{fig:errsVSparticles}
\end{figure}
Similar observations follow for the covariance error results presented in \Cref{fig:errsVSparticles}. For the left panel, which plots error in the posterior covariance against ensemble size, we notice that the ensemble covariance converges to the Bayesian posterior covariance (as proved in \Cref{cor:application_to_BIP}), and that reduced models with $r=10$ and $r=20$ have errors close to that of the full model. In the right panel, which shows posterior covariance error against reduced model size, we see that projection error dominates for small basis sizes and that sampling error dominates for larger basis sizes.
We remark that large ensemble sizes are needed to yield the accuracy presented in our plots, particularly for capturing the posterior covariance. This, however, is a typical challenge when approximating high-dimensional covariances from samples (in spectral norm) \cite{koltchinskii_concentration_2017}.
Overall, the results from \Cref{fig:errsVSparticles,fig:errsVSbasis} demonstrate that the proposed ensemble Kalman RMLE approach accurately estimates posterior statistics for linear Gaussian Bayesian inverse problems, and that the proposed strategy of accelerating ensemble Kalman RMLE with the balanced truncation approach of \Cref{subsec:BTBI} achieves negligible errors with an order-of-magnitude reduction in computational cost.

\section{Conclusion}\label{sec:Conclusions}

This work presents ensemble Kalman randomized maximum likelihood estimation (RMLE), a new derivative-free algorithm for randomized maximum likelihood estimation. The new method, which has connections to ensemble Kalman inversion, evolves an ensemble so that each ensemble member solves a randomized maximum likelihood estimation optimization problem. Linear analysis of the new method investigates convergence in two fundamental invariant subspaces for each the observation and state spaces. Theoretical results demonstrate that ensemble members converge exponentially fast to RMLE samples in the subspace that is observable and populated, and, when the method is applied to solve a suitably regularized optimization problem, ensemble members converge to samples from the Bayesian posterior.
The work also considers the problem of inferring the initial condition of a high-dimensional linear dynamical system, and shows how a balanced truncation model reduction approach can be used to accelerate the proposed method in this setting with negligible loss in accuracy. 

Numerical experiments on a randomly generated inverse problem verify linear theory, by illustrating convergence to the proposed infinite iteration ensemble mean and covariance limits. We also perform numerical experiments on a benchmark problem from the model reduction community that demonstrate convergence to the linear Bayesian posterior and reduced models that achieve accuracy close to the full model with an order-of-magnitude reduction in model size.

Although this work is mostly concerned with the linear case, we note that our method is also well-defined for nonlinear inverse problems. Understanding the algorithm's convergence in nonlinear and non-Gaussian problems is therefore a subject of study for future work. Other directions for future research include comparisons with other ensemble Kalman algorithms that sample from the Bayesian posterior, and development of methods for model reduction in nonlinear inverse problems.

\paragraph{Acknowledgements}
This work was supported by the United States
Air Force Office of Scientific Research (AFOSR) award FA9550-24-1-0105 (Program
Officer Dr. Fariba Fahroo).

\bibliographystyle{unsrt} 
\bibliography{main}

\begin{thebibliography}{10}

\bibitem{sanz-alonso_inverse_2023}
Daniel Sanz-Alonso, Andrew Stuart, and Armeen Taeb.
\newblock {\em Inverse {Problems} and {Data} {Assimilation}}.
\newblock London {Mathematical} {Society} {Student} {Texts}. Cambridge University Press, Cambridge, 2023.

\bibitem{bach_inverse_2024}
Eviatar Bach, Ricardo Baptista, Daniel Sanz-Alonso, and Andrew Stuart.
\newblock Inverse {Problems} and {Data} {Assimilation}: {A} {Machine} {Learning} {Approach}, October 2024.

\bibitem{rounce_quantifying_2020}
David~R. Rounce, Tushar Khurana, Margaret~B. Short, Regine Hock, David~E. Shean, and Douglas~J. Brinkerhoff.
\newblock Quantifying parameter uncertainty in a large-scale glacier evolution model using {Bayesian} inference: application to {High} {Mountain} {Asia}.
\newblock {\em Journal of Glaciology}, 66(256):175--187, April 2020.

\bibitem{yang_tipping_2020}
Fang Yang, Yayun Zheng, Jinqiao Duan, Ling Fu, and Stephen Wiggins.
\newblock The tipping times in an {Arctic} sea ice system under influence of extreme events.
\newblock {\em Chaos: An Interdisciplinary Journal of Nonlinear Science}, 30(6):063125, June 2020.

\bibitem{iglesias_well-posed_2014}
Marco~A Iglesias, Kui Lin, and Andrew~M Stuart.
\newblock Well-posed {Bayesian} geometric inverse problems arising in subsurface flow.
\newblock {\em Inverse Problems}, 30(11):114001, October 2014.

\bibitem{majumder_freshwater_2021}
Sudip Majumder, Renato~M. Castelao, and Caitlin~M. Amos.
\newblock Freshwater {Variability} and {Transport} in the {Labrador} {Sea} {From} {In} {Situ} and {Satellite} {Observations}.
\newblock {\em Journal of Geophysical Research: Oceans}, 126(4):e2020JC016751, 2021.

\bibitem{kaipio_statistical_2005}
Jari Kaipio and Erkki Somersalo.
\newblock {\em Statistical and computational inverse problems}.
\newblock Number v. 160 in Applied mathematical sciences. Springer, New York, 2005.

\bibitem{chung_numerical_2010}
Julianne Chung, James~G. Nagy, and Ioannis Sechopoulos.
\newblock Numerical {Algorithms} for {Polyenergetic} {Digital} {Breast} {Tomosynthesis} {Reconstruction}.
\newblock {\em SIAM Journal on Imaging Sciences}, 3(1):133--152, January 2010.

\bibitem{epstein_introduction_2007}
Charles~L. Epstein.
\newblock {\em Introduction to the {Mathematics} of {Medical} {Imaging}}.
\newblock Society for Industrial and Applied Mathematics, Philadelphia, PA, 2 edition, 2007.

\bibitem{scherzer_mathematical_2006}
Otmar Scherzer, Hans-Georg Bock, Frank De~Hoog, Avner Friedman, Arvind Gupta, Helmut Neunzert, William~R. Pulleyblank, Torgeir Rusten, Fadil Santosa, and Anna-Karin Tornberg, editors.
\newblock {\em Mathematical {Models} for {Registration} and {Applications} to {Medical} {Imaging}}, volume~10 of {\em Mathematics in industry}.
\newblock Springer, Berlin, Heidelberg, 2006.

\bibitem{tanaka_inverse_1993}
Masataka Tanaka and Huy~Duong Bui, editors.
\newblock {\em Inverse {Problems} in {Engineering} {Mechanics}: {IUTAM} {Symposium} {Tokyo}, 1992}.
\newblock Springer, Berlin, Heidelberg, 1993.

\bibitem{hansen_discrete_2010}
Per~Christian Hansen.
\newblock {\em Discrete {Inverse} {Problems}}.
\newblock Society for Industrial and Applied Mathematics, Philadelphia, PA, 2010.

\bibitem{mueller_linear_2012}
Jennifer~L. Mueller and Samuli Siltanen.
\newblock {\em Linear and {Nonlinear} {Inverse} {Problems} with {Practical} {Applications}}.
\newblock Society for Industrial and Applied Mathematics, Philadelphia, PA, 2012.

\bibitem{engl_regularization_2000}
Heinz~Werner Engl, Martin Hanke, and A.~Neubauer.
\newblock {\em Regularization of {Inverse} {Problems}}.
\newblock Springer Science \& Business Media, Dordrecht, March 2000.

\bibitem{hanke_regularizing_1997}
Martin Hanke.
\newblock Regularizing properties of a truncated newton-cg algorithm for nonlinear inverse problems.
\newblock {\em Numerical Functional Analysis and Optimization}, 18(9-10):971--993, January 1997.

\bibitem{akcelik_parallel_2002}
V.~Akcelik, G.~Biros, and O.~Ghattas.
\newblock Parallel {Multiscale} {Gauss}-{Newton}-{Krylov} {Methods} for {Inverse} {Wave} {Propagation}.
\newblock In {\em {SC} '02: {Proceedings} of the 2002 {ACM}/{IEEE} {Conference} on {Supercomputing}}, pages 41--41, November 2002.

\bibitem{petra_inexact_2012}
Noemi Petra, Hongyu Zhu, Georg Stadler, Thomas J.~R. Hughes, and Omar Ghattas.
\newblock An inexact {Gauss}-{Newton} method for inversion of basal sliding and rheology parameters in a nonlinear {Stokes} ice sheet model.
\newblock {\em Journal of Glaciology}, 58(211):889--903, January 2012.

\bibitem{epanomeritakis_newton-cg_2008}
I~Epanomeritakis, V~Akçelik, O~Ghattas, and J~Bielak.
\newblock A {Newton}-{CG} method for large-scale three-dimensional elastic full-waveform seismic inversion.
\newblock {\em Inverse Problems}, 24(3):034015, May 2008.

\bibitem{zhu_inversion_2016}
Hongyu Zhu, Noemi Petra, Georg Stadler, Tobin Isaac, Thomas J.~R. Hughes, and Omar Ghattas.
\newblock Inversion of geothermal heat flux in a thermomechanically coupled nonlinear {Stokes} ice sheet model.
\newblock {\em The Cryosphere}, 10(4):1477--1494, July 2016.

\bibitem{worthen_towards_2014}
Jennifer Worthen, Georg Stadler, Noemi Petra, Michael Gurnis, and Omar Ghattas.
\newblock Towards adjoint-based inversion for rheological parameters in nonlinear viscous mantle flow.
\newblock {\em Physics of the Earth and Planetary Interiors}, 234:23--34, September 2014.

\bibitem{conn_trust_2000}
Andrew~R. Conn, Nicholas I.~M. Gould, and Philippe~L. Toint.
\newblock {\em Trust {Region} {Methods}}.
\newblock Society for Industrial and Applied Mathematics, Philadelphia, PA, 2000.

\bibitem{dai_nonlinear_2011}
Yu-Hong Dai.
\newblock Nonlinear {Conjugate} {Gradient} {Methods}.
\newblock In {\em Wiley {Encyclopedia} of {Operations} {Research} and {Management} {Science}}. John Wiley \& Sons, Ltd, Hoboken, NJ, 2011.

\bibitem{ye_optimization_2019}
Nan Ye, Farbod Roosta-Khorasani, and Tiangang Cui.
\newblock Optimization {Methods} for {Inverse} {Problems}.
\newblock In Jan de~Gier, Cheryl~E. Praeger, and Terence Tao, editors, {\em 2017 {MATRIX} {Annals}}, pages 121--140. Springer International Publishing, Cham, 2019.

\bibitem{lu_stochastic_2022}
Shuai Lu and Peter Mathé.
\newblock Stochastic gradient descent for linear inverse problems in {Hilbert} spaces.
\newblock {\em Mathematics of Computation}, 91(336):1763--1788, July 2022.

\bibitem{jin_convergence_2020}
Bangti Jin, Zehui Zhou, and Jun Zou.
\newblock On the {Convergence} of {Stochastic} {Gradient} {Descent} for {Nonlinear} {Ill}-{Posed} {Problems}.
\newblock {\em SIAM Journal on Optimization}, 30(2):1421--1450, 2020.

\bibitem{sen_global_2013}
Mrinal~K. Sen and Paul~L. Stoffa.
\newblock {\em Global {Optimization} {Methods} in {Geophysical} {Inversion}}.
\newblock Cambridge University Press, Cambridge, 2013.

\bibitem{renyuan_combined_1996}
Tang Renyuan, Yang Shiyou, Li~Yan, Wen Geng, and Mei Tiemin.
\newblock Combined strategy of improved simulated annealing and genetic algorithm for inverse problem.
\newblock {\em IEEE Transactions on Magnetics}, 32(3):1326--1329, May 1996.

\bibitem{fernandez-martinez_particle_2010}
Juan~Luis Fernández-Martínez, Esperanza García-Gonzalo, and Véronique Naudet.
\newblock Particle swarm optimization applied to solving and appraising the streaming-potential inverse problem.
\newblock {\em GEOPHYSICS}, 75(4):WA3--WA15, July 2010.

\bibitem{calvello_ensemble_2024}
Edoardo Calvello, Sebastian Reich, and Andrew~M. Stuart.
\newblock Ensemble {Kalman} {Methods}: {A} {Mean} {Field} {Perspective}, October 2024.

\bibitem{chada_iterative_2021}
Neil~K. Chada, Yuming Chen, and Daniel Sanz-Alonso.
\newblock Iterative ensemble {Kalman} methods: {A} unified perspective with some new variants.
\newblock {\em Foundations of Data Science}, 3(3):331--369, April 2021.

\bibitem{iglesias_ensemble_2013}
Marco~A Iglesias, Kody J~H Law, and Andrew~M Stuart.
\newblock Ensemble {Kalman} methods for inverse problems.
\newblock {\em Inverse Problems}, 29(4):045001, April 2013.

\bibitem{emerick_ensemble_2013}
Alexandre~A. Emerick and Albert~C. Reynolds.
\newblock Ensemble smoother with multiple data assimilation.
\newblock {\em Computers \& Geosciences}, 55:3--15, June 2013.

\bibitem{emerick_investigation_2013}
Alexandre~A. Emerick and Albert~C. Reynolds.
\newblock Investigation of the sampling performance of ensemble-based methods with a simple reservoir model.
\newblock {\em Computational Geosciences}, 17(2):325--350, April 2013.

\bibitem{evensen_analysis_2018}
Geir Evensen.
\newblock Analysis of iterative ensemble smoothers for solving inverse problems.
\newblock {\em Computational Geosciences}, 22(3):885--908, June 2018.

\bibitem{gu_iterative_2007}
Yaqing Gu and Dean~S. Oliver.
\newblock An {Iterative} {Ensemble} {Kalman} {Filter} for {Multiphase} {Fluid} {Flow} {Data} {Assimilation}.
\newblock {\em SPE Journal}, 12(04):438--446, December 2007.

\bibitem{hanu_subsampling_2023}
Matei Hanu, Jonas Latz, and Claudia Schillings.
\newblock Subsampling in ensemble {Kalman} inversion.
\newblock {\em Inverse Problems}, 39(9):094002, July 2023.

\bibitem{bocquet_bayesian_2020}
Marc Bocquet, Julien Brajard, Alberto Carrassi, and Laurent Bertino.
\newblock Bayesian inference of chaotic dynamics by merging data assimilation, machine learning and expectation-maximization.
\newblock {\em Foundations of Data Science}, 2(1):55--80, March 2020.
\newblock Publisher: Foundations of Data Science.

\bibitem{cleary_calibrate_2021}
Emmet Cleary, Alfredo Garbuno-Inigo, Shiwei Lan, Tapio Schneider, and Andrew~M. Stuart.
\newblock Calibrate, emulate, sample.
\newblock {\em Journal of Computational Physics}, 424:109716, January 2021.

\bibitem{dunbar_calibration_2021}
Oliver R.~A. Dunbar, Alfredo Garbuno-Inigo, Tapio Schneider, and Andrew~M. Stuart.
\newblock Calibration and {Uncertainty} {Quantification} of {Convective} {Parameters} in an {Idealized} {GCM}.
\newblock {\em Journal of Advances in Modeling Earth Systems}, 13(9):e2020MS002454, 2021.

\bibitem{gottwald_supervised_2021}
Georg~A. Gottwald and Sebastian Reich.
\newblock Supervised learning from noisy observations: {Combining} machine-learning techniques with data assimilation.
\newblock {\em Physica D: Nonlinear Phenomena}, 423:132911, September 2021.

\bibitem{schneider_earth_2017}
Tapio Schneider, Shiwei Lan, Andrew Stuart, and João Teixeira.
\newblock Earth {System} {Modeling} 2.0: {A} {Blueprint} for {Models} {That} {Learn} {From} {Observations} and {Targeted} {High}-{Resolution} {Simulations}.
\newblock {\em Geophysical Research Letters}, 44(24):12,396--12,417, 2017.

\bibitem{katzfuss_understanding_2016}
Matthias Katzfuss, Jonathan~R. Stroud, and Christopher~K. Wikle.
\newblock Understanding the {Ensemble} {Kalman} {Filter}.
\newblock {\em The American Statistician}, 70(4):350--357, October 2016.

\bibitem{kalman_new_1960}
R.~E. Kalman.
\newblock A {New} {Approach} to {Linear} {Filtering} and {Prediction} {Problems}.
\newblock {\em Journal of Basic Engineering}, 82(1):35--45, March 1960.

\bibitem{qian_fundamental_2024}
Elizabeth Qian and Christopher Beattie.
\newblock The {Fundamental} {Subspaces} of {Ensemble} {Kalman} {Inversion}.
\newblock {\em \textit{To appear}, SIAM Review}, 2025.

\bibitem{schillings_analysis_2017}
Claudia Schillings and Andrew~M. Stuart.
\newblock Analysis of the {Ensemble} {Kalman} {Filter} for {Inverse} {Problems}.
\newblock {\em SIAM Journal on Numerical Analysis}, 55(3):1264--1290, January 2017.

\bibitem{schillings_convergence_2018}
C.~Schillings, , and A.~M. Stuart.
\newblock Convergence analysis of ensemble {Kalman} inversion: the linear, noisy case.
\newblock {\em Applicable Analysis}, 97(1):107--123, January 2018.

\bibitem{ding_ensemble_2021}
Zhiyan Ding and Qin Li.
\newblock Ensemble {Kalman} inversion: mean-field limit and convergence analysis.
\newblock {\em Statistics and Computing}, 31(1):9, January 2021.

\bibitem{blomker_well_2019}
Dirk Blömker, Claudia Schillings, Philipp Wacker, and Simon Weissmann.
\newblock Well posedness and convergence analysis of the ensemble {Kalman} inversion.
\newblock {\em Inverse Problems}, 35(8):085007, July 2019.

\bibitem{blomker_continuous_2022}
Dirk Blömker, Claudia Schillings, Philipp Wacker, and Simon Weissmann.
\newblock Continuous {Time} {Limit} of the {Stochastic} {Ensemble} {Kalman} {Inversion}: {Strong} {Convergence} {Analysis}.
\newblock {\em SIAM Journal on Numerical Analysis}, 60(6):3181--3215, December 2022.

\bibitem{blomker_strongly_2018}
Dirk Blömker, Claudia Schillings, and Philipp Wacker.
\newblock A {Strongly} {Convergent} {Numerical} {Scheme} from {Ensemble} {Kalman} {Inversion}.
\newblock {\em SIAM Journal on Numerical Analysis}, 56(4):2537--2562, January 2018.

\bibitem{chada_convergence_2022}
Neil Chada and Xin Tong.
\newblock Convergence acceleration of ensemble {Kalman} inversion in nonlinear settings.
\newblock {\em Mathematics of Computation}, 91(335):1247--1280, May 2022.

\bibitem{huang_efficient_2022}
Daniel~Zhengyu Huang, Jiaoyang Huang, Sebastian Reich, and Andrew~M Stuart.
\newblock Efficient derivative-free {Bayesian} inference for large-scale inverse problems.
\newblock {\em Inverse Problems}, 38(12):125006, December 2022.

\bibitem{armbruster_stabilization_2022}
Dieter Armbruster, Michael Herty, and Giuseppe Visconti.
\newblock A {Stabilization} of a {Continuous} {Limit} of the {Ensemble} {Kalman} {Inversion}.
\newblock {\em SIAM Journal on Numerical Analysis}, 60(3):1494--1515, June 2022.

\bibitem{tong_localized_2023}
X~T Tong and M~Morzfeld.
\newblock Localized ensemble {Kalman} inversion.
\newblock {\em Inverse Problems}, 39(6):064002, April 2023.

\bibitem{liu_dropout_2025}
Shuigen Liu, Sebastian Reich, and Xin~T. Tong.
\newblock Dropout {Ensemble} {Kalman} {Inversion} for {High} {Dimensional} {Inverse} {Problems}.
\newblock {\em SIAM Journal on Numerical Analysis}, 63(2):685--715, 2025.

\bibitem{al-ghattas_non-asymptotic_2023}
Omar Al-Ghattas and Daniel Sanz-Alonso.
\newblock Non-asymptotic analysis of ensemble {Kalman} updates: effective dimension and localization.
\newblock {\em Information and Inference: A Journal of the IMA}, 13(1):iaad043, December 2023.

\bibitem{chada_tikhonov_2020}
Neil~K. Chada, Andrew~M. Stuart, and Xin~T. Tong.
\newblock Tikhonov {Regularization} within {Ensemble} {Kalman} {Inversion}.
\newblock {\em SIAM Journal on Numerical Analysis}, 58(2):1263--1294, January 2020.

\bibitem{stuart_inverse_2010}
A.~M. Stuart.
\newblock Inverse problems: {A} {Bayesian} perspective.
\newblock {\em Acta Numerica}, 19:451--559, May 2010.

\bibitem{dashti_map_2013}
M~Dashti, K~J~H Law, A~M Stuart, and J~Voss.
\newblock {MAP} estimators and their consistency in {Bayesian} nonparametric inverse problems.
\newblock {\em Inverse Problems}, 29(9):095017, September 2013.

\bibitem{garbuno-inigo_interacting_2020}
Alfredo Garbuno-Inigo, Franca Hoffmann, Wuchen Li, and Andrew~M. Stuart.
\newblock Interacting {Langevin} {Diffusions}: {Gradient} {Structure} and {Ensemble} {Kalman} {Sampler}.
\newblock {\em SIAM Journal on Applied Dynamical Systems}, 19(1):412--441, January 2020.

\bibitem{nusken_note_2019}
Nikolas Nüsken and Sebastian Reich.
\newblock Note on {Interacting} {Langevin} {Diffusions}: {Gradient} {Structure} and {Ensemble} {Kalman} {Sampler} by {Garbuno}-{Inigo}, {Hoffmann}, {Li} and {Stuart}, August 2019.

\bibitem{roberts_exponential_1996}
Gareth~O. Roberts and Richard~L. Tweedie.
\newblock Exponential convergence of {Langevin} distributions and their discrete approximations.
\newblock {\em Bernoulli}, 2(4):341--363, December 1996.

\bibitem{huang_iterated_2022}
Daniel~Zhengyu Huang, Tapio Schneider, and Andrew~M. Stuart.
\newblock Iterated {Kalman} methodology for inverse problems.
\newblock {\em Journal of Computational Physics}, 463:111262, August 2022.

\bibitem{peterson_mean_1987}
Carsten Peterson and James~R Anderson.
\newblock A mean field theory learning algorithm for neural networks.
\newblock {\em Complex Systems}, 1:995--1019, 1987.

\bibitem{opper_variational_2009}
Manfred Opper and Cédric Archambeau.
\newblock The {Variational} {Gaussian} {Approximation} {Revisited}.
\newblock {\em Neural Computation}, 21(3):786--792, March 2009.

\bibitem{marzouk_sampling_2016}
Youssef Marzouk, Tarek Moselhy, Matthew Parno, and Alessio Spantini.
\newblock Sampling via {Measure} {Transport}: {An} {Introduction}.
\newblock In Roger Ghanem, David Higdon, and Houman Owhadi, editors, {\em Handbook of {Uncertainty} {Quantification}}, pages 1--41. Springer International Publishing, Cham, 2016.

\bibitem{reich_dynamical_2011}
Sebastian Reich.
\newblock A dynamical systems framework for intermittent data assimilation.
\newblock {\em BIT Numerical Mathematics}, 51(1):235--249, March 2011.

\bibitem{el_moselhy_bayesian_2012}
Tarek~A. El~Moselhy and Youssef~M. Marzouk.
\newblock Bayesian inference with optimal maps.
\newblock {\em Journal of Computational Physics}, 231(23):7815--7850, October 2012.

\bibitem{geyer_practical_1992}
Charles~J. Geyer.
\newblock Practical {Markov} {Chain} {Monte} {Carlo}.
\newblock {\em Statistical Science}, 7(4):473--483, November 1992.

\bibitem{roberts_weak_1997}
G.~O. Roberts, A.~Gelman, and W.~R. Gilks.
\newblock Weak {Convergence} and {Optimal} {Scaling} of {Random} {Walk} {Metropolis} {Algorithms}.
\newblock {\em The Annals of Applied Probability}, 7(1):110--120, 1997.

\bibitem{goodman_ensemble_2010}
Jonathan Goodman and Jonathan Weare.
\newblock Ensemble samplers with affine invariance.
\newblock {\em Communications in Applied Mathematics and Computational Science}, 5(1):65--80, January 2010.

\bibitem{bardsley_mcmc-based_2012}
Johnathan~M. Bardsley.
\newblock {MCMC}-{Based} {Image} {Reconstruction} with {Uncertainty} {Quantification}.
\newblock {\em SIAM Journal on Scientific Computing}, 34(3):A1316--A1332, January 2012.

\bibitem{beskos_sequential_2015}
Alexandros Beskos, Ajay Jasra, Ege~A. Muzaffer, and Andrew~M. Stuart.
\newblock Sequential {Monte} {Carlo} methods for {Bayesian} elliptic inverse problems.
\newblock {\em Statistics and Computing}, 25(4):727--737, July 2015.

\bibitem{del_moral_sequential_2006}
Pierre Del~Moral, Arnaud Doucet, and Ajay Jasra.
\newblock Sequential {Monte} {Carlo} {Samplers}.
\newblock {\em Journal of the Royal Statistical Society Series B: Statistical Methodology}, 68(3):411--436, June 2006.

\bibitem{chen_ensemble_2012}
Yan Chen and Dean~S. Oliver.
\newblock Ensemble {Randomized} {Maximum} {Likelihood} {Method} as an {Iterative} {Ensemble} {Smoother}.
\newblock {\em Mathematical Geosciences}, 44(1):1--26, January 2012.

\bibitem{bardsley_randomize-then-optimize_2014}
Johnathan~M. Bardsley, Antti Solonen, Heikki Haario, and Marko Laine.
\newblock Randomize-{Then}-{Optimize}: {A} {Method} for {Sampling} from {Posterior} {Distributions} in {Nonlinear} {Inverse} {Problems}.
\newblock {\em SIAM Journal on Scientific Computing}, 36(4):A1895--A1910, January 2014.

\bibitem{benner_survey_2015}
Peter Benner, Serkan Gugercin, and Karen Willcox.
\newblock A {Survey} of {Projection}-{Based} {Model} {Reduction} {Methods} for {Parametric} {Dynamical} {Systems}.
\newblock {\em SIAM Review}, 57(4):483--531, January 2015.

\bibitem{antoulas_approximation_2005}
Athanasios~C. Antoulas.
\newblock {\em Approximation of {Large}-{Scale} {Dynamical} {Systems}}.
\newblock Society for Industrial and Applied Mathematics, Philadelphia, PA, 2005.

\bibitem{antoulas_interpolatory_2020}
A.~C. Antoulas, C.~A. Beattie, and S.~Gügercin.
\newblock {\em Interpolatory {Methods} for {Model} {Reduction}}.
\newblock Computational {Science} \& {Engineering}. Society for Industrial and Applied Mathematics, Philadelphia, PA, 2020.

\bibitem{willcox_balanced_2002}
K.~Willcox and J.~Peraire.
\newblock Balanced {Model} {Reduction} via the {Proper} {Orthogonal} {Decomposition}.
\newblock {\em AIAA Journal}, 40(11):2323--2330, November 2002.

\bibitem{lipponen_electrical_2013}
Antti Lipponen, Aku Seppanen, and Jari~P. Kaipio.
\newblock Electrical impedance tomography imaging with reduced-order model based on proper orthogonal decomposition.
\newblock {\em Journal of Electronic Imaging}, 22(2):023008, May 2013.

\bibitem{lieberman_parameter_2010}
Chad Lieberman, Karen Willcox, and Omar Ghattas.
\newblock Parameter and {State} {Model} {Reduction} for {Large}-{Scale} {Statistical} {Inverse} {Problems}.
\newblock {\em SIAM Journal on Scientific Computing}, 32(5):2523--2542, 2010.

\bibitem{nguyen_model_2014}
Ngoc-Hien Nguyen, Boo Cheong~Khoo, and Karen Willcox.
\newblock Model order reduction for {Bayesian} approach to inverse problems.
\newblock {\em Asia Pacific Journal on Computational Engineering}, 1(1):2, April 2014.

\bibitem{qian_certified_2017}
Elizabeth Qian, Martin Grepl, Karen Veroy, and Karen Willcox.
\newblock A {Certified} {Trust} {Region} {Reduced} {Basis} {Approach} to {PDE}-{Constrained} {Optimization}.
\newblock {\em SIAM Journal on Scientific Computing}, 39(5):S434--S460, 2017.

\bibitem{karcher_reduced_2018}
Mark Kärcher, Sébastien Boyaval, Martin~A. Grepl, and Karen Veroy.
\newblock Reduced basis approximation and a posteriori error bounds for {4D}-{Var} data assimilation.
\newblock {\em Optimization and Engineering}, 19(3):663--695, September 2018.

\bibitem{qian_model_2022}
Elizabeth Qian, Jemima~M. Tabeart, Christopher Beattie, Serkan Gugercin, Jiahua Jiang, Peter~R. Kramer, and Akil Narayan.
\newblock Model {Reduction} of {Linear} {Dynamical} {Systems} via {Balancing} for {Bayesian} {Inference}.
\newblock {\em Journal of Scientific Computing}, 91(1):29, March 2022.

\bibitem{konig_time-limited_2023}
Josie König and Melina~A. Freitag.
\newblock Time-{Limited} {Balanced} {Truncation} for {Data} {Assimilation} {Problems}.
\newblock {\em Journal of Scientific Computing}, 97(2):47, October 2023.

\bibitem{freitag_inference-oriented_2024}
Melina~A. Freitag, Josie König, and Elizabeth Qian.
\newblock Inference-{Oriented} {Balanced} {Truncation} for {Quadratic} {Dynamical} {Systems}: {Formulation} for {Bayesian} {Smoothing} and {Model} {Stability} {Analysis}.
\newblock {\em PAMM}, 24(4):e202400051, 2024.

\bibitem{gugercin_survey_2004}
Serkan Gugercin, , and Athanasios~C. Antoulas.
\newblock A {Survey} of {Model} {Reduction} by {Balanced} {Truncation} and {Some} {New} {Results}.
\newblock {\em International Journal of Control}, 77(8):748--766, May 2004.

\bibitem{moore_principal_1981}
B.~Moore.
\newblock Principal component analysis in linear systems: {Controllability}, observability, and model reduction.
\newblock {\em IEEE Transactions on Automatic Control}, 26(1):17--32, February 1981.

\bibitem{spantini_optimal_2015}
Alessio Spantini, Antti Solonen, Tiangang Cui, James Martin, Luis Tenorio, and Youssef Marzouk.
\newblock Optimal {Low}-rank {Approximations} of {Bayesian} {Linear} {Inverse} {Problems}.
\newblock {\em SIAM Journal on Scientific Computing}, 37(6):A2451--A2487, January 2015.

\bibitem{zahm_certified_2022}
Olivier Zahm, Tiangang Cui, Kody Law, Alessio Spantini, and Youssef Marzouk.
\newblock Certified dimension reduction in nonlinear {Bayesian} inverse problems.
\newblock {\em Mathematics of Computation}, 91(336):1789--1835, July 2022.

\bibitem{koltchinskii_concentration_2017}
Vladimir Koltchinskii and Karim Lounici.
\newblock Concentration inequalities and moment bounds for sample covariance operators.
\newblock {\em Bernoulli}, 23(1):110--133, February 2017.

\end{thebibliography}


\end{document}